\documentclass{amsart}

\usepackage{amssymb, amsmath,  amsfonts,graphicx}%
\usepackage{color }

\usepackage[all,cmtip]{xy}
\newtheorem{thm}{Theorem}[section]
\newtheorem{prop}[thm]{Proposition}
\newtheorem{lemma}[thm]{Lemma}
\newtheorem{cor}[thm]{Corollary}
\newtheorem{defn}[thm]{Definition}
\newtheorem{example}[thm]{Example}

\newtheorem{prop-defn}[thm]{Proposition/Definition}

\newtheorem{remark}[thm]{Remark}

\newtheorem{qus}[thm]{Question}

\newtheorem{conjecture}[thm]{Conjecture}

\begin{document}{\allowdisplaybreaks[4]

\title{$W$-translated Schubert divisors and transversal intersections}

\author{DongSeon Hwang }
\address{Department of Mathematics,
           Ajou University,  San 5, Woncheon-dong, Yeongtong-gu, Suwon 443-749, Korea}
\email{dshwang@ajou.ac.kr}
\thanks{ 
 }

\author{Hwayoung Lee }
\address{Department of Mathematics, University of California, Riverside, CA92521, USA}

\email{hwayoung.lee@ucr.edu}
\thanks{ 
 }

\author{Jae-Hyouk Lee }
\address{Department of Mathematics, Ewha Womans University, 11-1, Daehyun-Dong, Seodaemun-Gu, Seoul   120-750, Korea}
\email{jaehyoukl@ewha.ac.kr}
\thanks{ 
 }

\author{Changzheng Li}
 \address{School of Mathematics, Sun Yat-sen University, Guangzhou 510275, P.R. China;}
\email{lichangzh@mail.sysu.edu.cn}

\date{
      }




\begin{abstract}
We study the toric degeneration of Weyl group translated Schubert divisors of a partial flag variety $F\ell_{n_1, n_2, \cdots, n_k; n}$ via Gelfand-Cetlin polytopes. We propose a conjecture  that  Schubert varieties  of appropriate dimensions intersect transversally up to translation by Weyl group elements, and verify it in various cases, including complex Grassmannian $Gr(2, n)$ and complete flag variety $F\ell_{1,2,3; 4}$.
  \end{abstract}

\maketitle
 \tableofcontents

\section{Introduction}
 The partial flag variety $X=F\ell_{n_1, \cdots, n_k; n}$, parameterizing the space of partial flags $\{V_{n_1}\subseteq\cdots \subseteq V_{n_k}\subseteq \mathbb{C}^n\}$, is a smooth projective  $SL(n, \mathbb{C})$-homogeneous variety. The cohomology ring $H^*(X, \mathbb{Z})$ is torsion-free, and has a $\mathbb{Z}$-additive basis of Schubert classes $\sigma^u$ labelled  by a subset of the permutation group $S_n$ of $n$ objects.
The structure constants $N_{u_1,\cdots, u_m, u_{m+1}}^w$ in the cup product
   $$\sigma^{u_1}\cup\cdots\cup  \sigma^{u_{m+1}}=\sum_{ \ell(w)=\ell(u_1)+\cdots+\ell(u_{m+1})} N_{u_1,\cdots, u_{m+1}}^w\sigma^w$$
 are nonnegative integers, where $\ell: S_n\to \mathbb{Z}_{\geq 0}$ denotes the standard length function.
 One of the most central problems in Schubert calculus is finding a manifestly positive combinatorial formula of $N_{u_1, u_2}^w$, which determines a manifestly positive  formula for general  $N_{u_1,\cdots, u_{m+1}}^w$. However, this problem is widely open except for   two-step flag varieties $F\ell_{n_1, n_2; n}$ \cite{Coskun-twostep, BKPT,Buch-equivTwostep}, beyond the famous Littlewood-Ricardson rules for complex Grassmannians $Gr(m, n)=F\ell_{m;n}$. In addition, there is  a recent work \cite{KnZJ} on three-step flag varieties, and there has been a nice algorithm for the Schubert structure constants \cite{Duan} (with sign cancelation involved) even for flag varieties of general Lie type.    In the present paper, we propose a geometric approach towards
   a manifestly positive formula of $N_{u, v}^w$ for $F\ell_{n_1, \cdots, n_k; n}$.

 Let us start with the geometrical  meaning of the Schubert structure constants $N_{u, v}^w$. Inside $X$, there are   Schubert varieties $X_w$ of  complex  dimension $\ell(w)$, which are closed subvarieties of $X$ satisfying
 appropriate dimension conditions on  the  intersections of partial flags with an a priori fixed  complete flag in $\mathbb{C}^n$. There are opposite Schubert varieties $X^w$ of complex codimension $\ell(w)$ defined by imposing dimension conditions with respect to the complete flag opposite to the fixed one.
 Kleiman's Transversality Theorem \cite{Kleiman} says that  for generic $g, g'\in SL(n, \mathbb{C})$,   the varieties
    $g\cdot X^u, g'\cdot X^v$ and $ X_w$ intersect  transversally along a Zariski dense open subset of every component of their intersection. Moreover,       the transversal intersection  $g\cdot X^u\cap g'\cdot X^v\cap   X_w$ is a finite set of cardinality $N_{u, v}^w$, provided $\ell(w)=\ell(u)+\ell(v)$.
Along this direction, one may ask the following   question.
  \begin{qus}\label{question1}
   How to specify    a pair $(g, g')$ of elements in $SL(n, \mathbb{C})$ that satisfies the Kleiman's transversality for given $u, v, w$?
  \end{qus}
\noindent  It would be great if one could   characterize   the pair $(g, g')$ of generic elements and give a combinatorial description of the transversal intersection, which would result in a  manifestly positive   formula of $N_{u, v}^w$.
 The desired pairs $(g, g')$ exist almost everywhere. However, it is still like looking for a needle in a haystack if one really wants to specify one such $(g, g')$.
In this paper, we consider the   finitely many subvarieties $uX^v$ where $u\in S_n$, and call them   \textit{$W$-translated Schubert varieties}.
 We will study the transversality of intersections among them.  We remark that varieties of the form $g_1X^{u_1}\bigcap\cdots \bigcap g_mX^{u_m}$ for $g_1,\cdots, g_m\in SL(n, \mathbb{C})$ are called \textit{intersection varieties}
in \cite{BiCo}, whose singular loci were studied therein.

Our study of transversal intersections of $W$-translated Schubert varieties uses toric degenerations and Gelfand-Cetlin polytopes. This  was   inspired by  the  works \cite{Kiri}, \cite{KST} of Kiritchenko,  Smirnov and   Timorin, which     related Schubert calculus with Gelfand-Cetlin polytopes closely in an algebraic/representation-theoretic way, in addition to the early study in \cite{Kogan}.
 Algebraically,    the polytope ring \cite{PuKh,Timo} associated to the Gelfand-Cetlin polytope for a complete flag variety  $F\ell_{1,2,\cdots, n-1; n}$ is isomorphic to  $H^*(F\ell_{1,2,\cdots, n-1; n}, \mathbb{Z})$ \cite{Kaveh};
  descriptions of the Schubert classes modelled by such Gelfand-Cetlin polytope  through the associated polytope ring are given in    \cite{Kiri, KST}.
  Geometrically,
there are  toric degenerations of partial flag varieties to the toric varieties $X_\Delta$ of the Gelfand-Cetlin polytopes $\Delta$ \cite{GoLa,BCFKS,KoMi,KnMi, NNU}.
The standard (resp. opposite) Schubert varieties of the partial flag variety $X$ behave well with respect to the toric degeneration, in the sense that they degenerate to reduced unions of toric subvarieties of $X_\Delta$ \cite[Theorem 8 and Remark 10]{KoMi}.   The $W$-translated Schubert varieties  are  also expected to enjoy similar   nice degeneration properties. This leads to the strategy   that  \textit{ intersection problems in the generic fiber are likely transferred to  intersection problems of toric subvarieties of $X_\Delta$, and hence to that of faces of the Gelfand-Cetlin polytope $\Delta$}.  Indeed, as we will show in {Theorem  \ref{thmToricDeg}}, every $W$-translated Schubert divisor, which is defined by a positive path,
degenerates to the union of toric divisors, defined by effective edges that lie on the positive path.
 Each $W$-translated Schubert variety $uX^v$ is a scheme-theoretical intersections of $X$ with appropriate   $W$-translated Schubert divisors. Therefore it corresponds to a set-theoretic  intersection $\Delta(u, v)$ of unions of faces of $\Delta$ of codimension one. The above strategy for $X=F\ell_n$ can be achieved in the following way,
 which  builds a bridge   connecting geometry and combinatorics  with the help of toric degeneration.
 \begin{thm}\label{thmTransInterviatoric00}
    Let $v_1, \cdots, v_{m+1}, w\in S_n$ with $\ell(w)=\sum_{i=1}^{m+1}\ell(v_i)$, where $m\geq 1$. Suppose that there exist
    $u_1,\cdots, u_{m+1}\in S_n$ with $\mathcal{S}:=\bigcap_{i=1}^{m+1} \Delta(u_i, v_i) \bigcap \Delta(w_0, w_0w)$  consisting of part of the regular vertices of  $\Delta$ for $F\ell_n$.   Then  $u_1 X^{v_1},\cdots, u_{m+1}X^{v_{m+1}}$,  $X_w$ intersect transversally, and  $N_{v_1,\cdots, v_{m+1}}^w=\sharp \mathcal{S}$.
\end{thm}
\noindent Here $w_0$ denotes the longest element of $S_n$, and we refer to  \textbf{Theorem \ref{thmTransInterviatoric}} for more precise statements for general $F\ell_{n_1, \cdots, n_k; n}$.
  We notice that all $W$-translated Schubert varieties $\{wX^u\}_{w\in W}$ represent the same Schubert class $\sigma^u$.
  In \cite{Kiri}, Kiritchenko assigned  to $\sigma^u$ different choices of faces of the Gelfand-Cetlin polytopes $\Delta$ for complete flag varieties in a  representation-theoretic way.
  Instead of the one-to-many correspondence between Schubert classes and faces of the polytope $\Delta$, we believe that it is more natural  to make a one-to-one correspondence between $W$-translated Schubert varieties and combinations of faces of the polytope, arising from the toric degeneration. This will probably fill the gap between geometry and the algebraic model in \cite{KST}. For instance, Remark 2.5 therein can be explained in this way (namely the toric divisor of the  combination of the two faces in the remark is the degeneration of the corresponding $W$-Schubert divisor, which should be treated as one single object since the very beginning from the viewpoint of geometry).
For the case of Schubert divisor classes, we will describe the expected  relationship precisely   in Remark \ref{geomexp}.

To explore an answer to Question \ref{question1}, we will construct a partition $\mathcal{U}$ of the set  $\{({u, v},w)\mid u, v, w\in S_n, \ell(w)=\ell(u)+\ell(v)\}$  in section \ref{covering}. Each triple $(u, v, w)$ represents a Schubert structure constant $N_{u, v}^w$. Therein we will also modify the partition a little bit to obtain $\hat {\mathcal{U}}$,  by adding elements of the form   $(u_1',\cdots, u_{m+1}', {w'})$ to each class $[({u, v},{w})]\in \mathcal{U}$. As we will see, the way of constructing $\hat{\mathcal{U}}$ ensures that Schubert structure constants represented by elements in a same class   are equal to each other, namely we have $N_{u, v}^w=N_{u_1',\cdots, u_{m+1}'}^{w'}$.
We make the following conjecture.
\begin{conjecture}\label{mainconj00}
   For any  $[({u, v},{w})]\in \hat{\mathcal{U}}$,  there exists $(u_1',\cdots, u_{m+1}',{w'})\in [({u, v},{w})]$ such that
     $\tilde u_1\cdot X^{u'_1}\cap\cdots\cap\tilde u_{m+1} \cdot  X^{u'_{m+1}}\cap X_{w'}$ is a  transversal intersection   for some
    $\tilde u_1,\cdots, \tilde u_{m+1}\in S_n$.

\end{conjecture}
\noindent We remark that $m=1$ is usually taken,  while $m>1$ is sometimes   necessary as we will see in the example for $Gr(3, 6)$  in section 4.4.
 If  Conjecture \ref{mainconj00} is true, then we will obtain a Littlewood-Richardson rule by  {Theorem \ref{thmTransInterviatoric}}. On one hand, we could reduce the searching of $(g, g')$ in a Zariski open dense subset of $SL(n, \mathbb{C})^2$ in Question \ref{question1} to a small pool of finitely many possibilities, provided the conjecture held. On the other hand, the number of $W$-translated Schubert divisors that define  a $W$-translated Schubert variety   is strictly greater than the codimension of such variety in general. As a price, the study of the corresponding  set-theoretic intersections of faces of the Gelfand-Cetlin polytope is simple in principle, but complicated in practice.
 We will study some concrete cases in section 4, and obtain the following.
\begin{thm}
   Conjecture \ref{mainconj00} holds if the Schubert structure constant $N_{u, v}^w$ indexed  by $(u, v, w)$  is either {\upshape $(\mbox{i})$} a Chevalley type structure constant for $Gr(m, n)$ or  {\upshape $(\mbox{ii})$} an arbitrary structure constant for $Gr(2, n)$ or $F\ell_{4}$.
\end{thm}
\noindent We remark that there have also been earlier works of  Schubert structure constants for Grassmannians by studying the transversality of intersections of Schubert varieties
    in different  ways  \cite{Sott-transversal,Sott}  or by different degeneration methods   \cite{Vakil}.

As one byproduct of the study of toric degeneration,  we give explicit defining equations for certain toric subvarieties of $X_\Delta$, including all the toric divisors in \textbf{Proposition \ref{proptoricdivisor}}.
As another byproduct, we  give a combinatorial description of    an   anti-canonical divisor $-K_X$ of $X$ in \textbf{Proposition \ref{propAntiCano}}, using the notions of special paths. Such anti-canonical divisor degenerates to that of the Gelfand-Cetlin toric variety $X_0$ and
 is potentially useful in the study of    Strominger-Yau-Zaslow
 mirror symmetry \cite{SYZ} for a partial flag variety $X$. For instance, Chan, Leung and the fourth named author constructed a special Lagrangian fibration for the open Calabi-Yau
 manifold $X\setminus -K_X$ for the two-step flag variety $F\ell_{1, n-1;n}$ \cite{CLL}, which is one key step in the study of SYZ mirror symmetry.

The present paper is organized as follows. In  section 2, we review basic facts on  Gelfand-Cetlin toric varieties.
 In section 3, we introduce the notion of $W$-translated Schubert varieties, and study the toric degeneration of certain $W$-translated Schubert varieties.
 In section 4, we   study  transversal   intersections of $W$-translated Schubert varieties by using toric degenerations in various examples. Finally in section 5, we specify an anti-canonical divisor of $X$.

\section{Gelfand-Cetlin  polytopes  and toric  varieties}

Throughout this paper, we fix a sequence  $0=n_0<  n_1<n_2<\cdots<n_k<n_{k+1}=n$ of integers, together with
a decreasing  sequence of nonnegative integers:
$$\lambda=\{\lambda^{(n)}_1=\lambda^{(n)}_2=\cdots =\lambda^{(n)}_{n_1}>\lambda^{(n)}_{n_1+1}=\cdots=\lambda^{(n)}_{n_2}>\cdots>\lambda^{(n)}_{n_k+1}=\cdots =\lambda^{(n)}_n\}. $$
 In this section, we  review   notions and facts on Gelfand-Cetlin polytopes and toric varieties. We mainly follow    \cite{BCFKS,KoMi} while give the descriptions in a uniform language by ladder diagrams and positive paths, which  will take   advantages in the later study of transversal intersections.
In what follows, a unit square is called a box.

\begin{defn}
   Let  $Q$ be an $n\times n$ square, i.e. a square consisting of $n^2$ boxes, and $Q_l$ be  squares   of size $(n_l-n_{l-1})\times (n_l-n_{l-1}) \, (l=1, 2, \cdots, k+1)$ placed on the diagonal of $Q$ in the lower right direction. The  \textbf{ladder diagram} $\Lambda:=\Lambda(n_1, \cdots, n_k; n)$ is the set of boxes   below  the diagonal squares $Q_l$'s.
    \end{defn}

 \begin{figure}[h]
  \caption{The ladder diagram $\Lambda$ and its associated dual graph $\Gamma$}\label{LDforFl}
  \bigskip
      \includegraphics[scale=1]{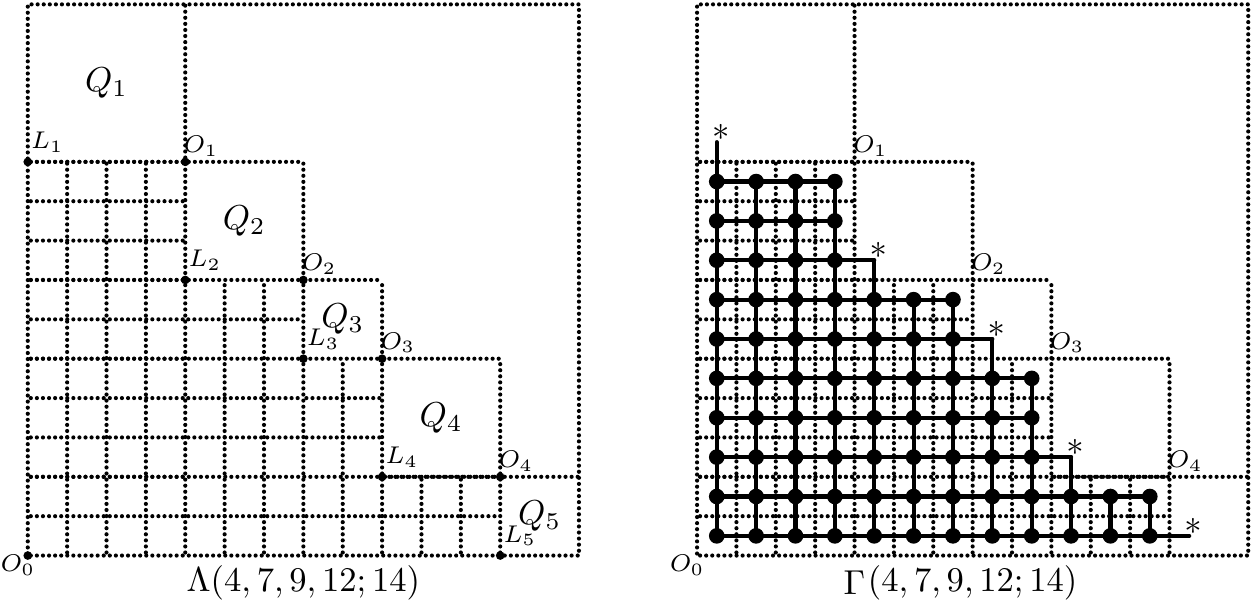}
  \end{figure}

 For each index $l$, we denote the lower right (resp. lower left)   vertex  of $Q_l$ as $O_l$ (resp. $L_l$) as in  Figure \ref{LDforFl}. We denote by   $O_0$  the lower left   vertex   of $\Lambda$ (or of $Q$).

 \begin{defn}
  We refer to the upper right side of the boundary
    of $\Lambda$ as the \textbf{roof} of $\Lambda$. An edge of a box is said to be   \textbf{effective}, if it is either  in the interior of $\Lambda$  or  on the roof     meeting some  vertex $L_i$.
   \end{defn}

\begin{remark}
One can associate, to a ladder diagram $\Lambda$, a graph $\Gamma$ by using so-called extended plane duality introduced in \cite[Definition 2.1.3]{BCFKS}.
\end{remark}

Each effective edge of $\Lambda$ defines a facet of the associated Gelfand-Cetlin polytope  to be described below. Recall that $\{\lambda^{(n)}_j\}_j$ are fixed already.
\begin{defn}
An array $\{\lambda_j^{(i)}\}_{1\leq j\leq i\leq n}$ of real numbers is called a Gelfand-Cetlin pattern for $\lambda$ if  all the inequalities $\lambda^{(i+1)}_j\geq \lambda_j^{(i)}\geq \lambda^{(i+1)}_{j+1}$ hold for $1\leq j\leq i\leq n-1$.
The associated Gelfand-Cetlin polytope, denoted as $\Delta_\lambda$ or simply   $\Delta$, is the convex hull of all integer Gelfand-Cetlin patterns for $\lambda$ in $\mathbb{R}^{n(n+1)\over 2}$.
 \end{defn}
   Every Gelfand-Cetlin pattern can be put in the boxes below the diagonal boxes of $Q$. 
    Precisely, we assign  $\lambda_j^{(i)}$ in a   box along the direction as illustrated in Figure  \ref{GCpic}. In particular, all the boxes inside   square $Q_l$ are assigned to the constant quantity $\lambda^{(n)}_{n_l}$ for each $1\leq l\leq k+1$, due to the inequalities. We have  $$\dim_\mathbb{R}\Delta=\sum_{i=1}^k(n_i-n_{i-1})(n-n_i),$$ which  equals the number of boxes in the ladder diagram $\Lambda(n_1, \cdots, n_k; n)$.
  \begin{figure}[h]
    \caption{Gelfand-Cetlin patterns for $\Lambda{(3,5;8)}$}\label{GCpic}
    \includegraphics[scale=1]{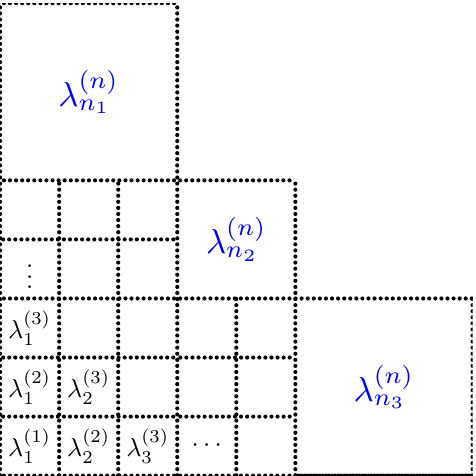}
  \end{figure}

 \noindent Clearly,  the following are equivalent:
   \begin{enumerate}
     \item An edge  $e$ in $\Lambda$   is effective;
     \item Each edge $e$ of $\Lambda$ presents an equation $a_e = b_e$, where $a_e$, $b_e$ are the assigned quantities of the two boxes or squares containing $e$.  Here, the equation $a_e=b_e$ defines a facet $F_e$ of $\Delta $ (i.e., $F_e:=\Delta \cap \{a_e=b_e\}$ is a face of codimension one).
   \end{enumerate}
The polytope $\Delta$ defines a normal toric variety $X_{\Delta}$ together with its projective embedding, the so-called  \textit{Gelfand-Cetlin   toric variety}. Moreover, it is well known that faces $F$ of $\Delta$ are in one-to-one correspondence with toric subvarieties $X_{\Delta(F)}$ of $X_{\Delta}$. Here by a toric subvariety, we always assume it to be reduced and   irreducible, and refer to \cite{CLS} for the background on toric varieties.
\subsubsection*{Defining equations of $X_\Delta$}
The Gelfand-Cetlin toric variety $X_\Delta$ can be described in terms of explicit defining equations, as we will review below.
\begin{defn}
 A \textbf{positive path}  $\pi$   on the ladder diagram $\Lambda(n_1,\cdots, n_k; n)$ is a path  starting   at the lower left vertex $O_0$ and moving either upward   or to the right  along edges, towards  one of $O_j, 1\leq j\leq k$. We define a partial order on the set of  positive paths    as follows:      ${\pi'}\leq {\pi}$  if    $\pi$ runs above   $\pi'$.

 A pair $(\pi, {\pi'})$ is called  {incomparable}, if neither $\pi\leq {\pi'}$ nor ${\pi'}\leq \pi$ holds.
 \end{defn}

We denote by $\pi_{i_1, \cdots, i_{l}}$   the positive path $\pi$ being horizontal exactly at the $i_1$, $i_2, \cdots, i_{l}$-th steps.
 Given   $I=[i_1, \cdots, i_\ell]$ and $J=[j_1, \cdots, j_m]$, it follows   from the definition that   $\pi_I\leq \pi_J$ if and only if $\ell\geq m$ and $i_r\leq j_r$ for
 $r=1, \cdots, m.$
 \begin{example}
   For $\Lambda({4,7,9,12; 14})$ as in Figure \ref{PPforFl}, we have  ${\pi}=\pi_{4,5,8,9,10, 13,14}$,
${\pi'}=\pi_{1,2,5,6,7,8,11,12,14}$, and ${\pi''}=\pi_{1,2,3,4,5,6,7,8,14}$.  Moreover,   ${\pi}\geq {\pi'}\geq {\pi''}$.
\end{example}

\begin{figure}[h]
  \caption{Positive paths and the associated exponent vectors}\label{PPforFl}
  \bigskip
     \includegraphics[scale=0.8]{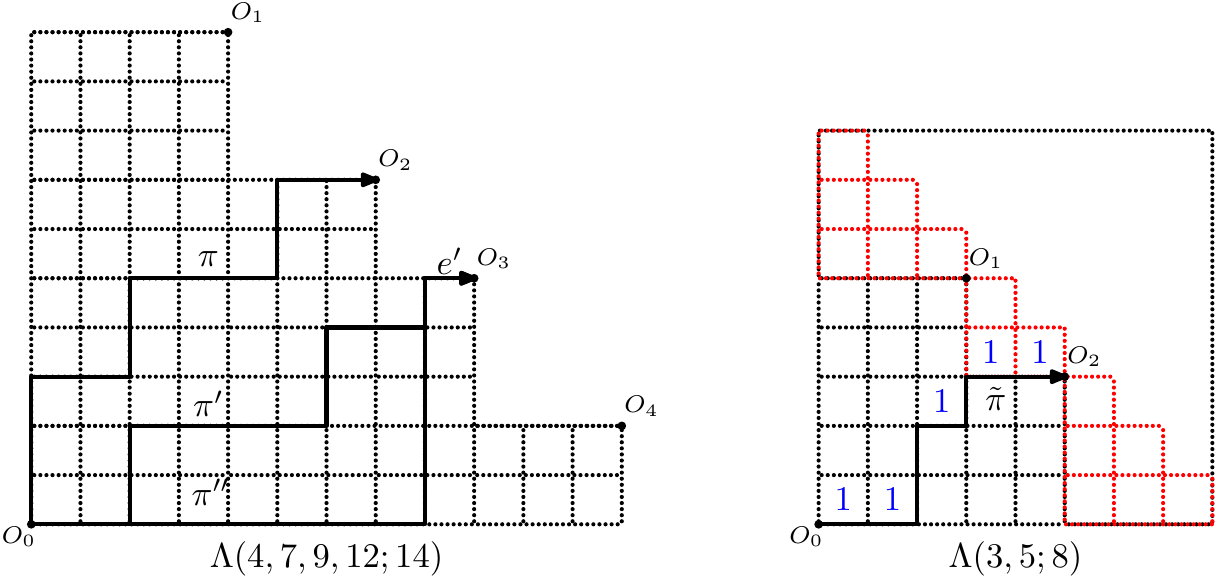}
  \end{figure}

\begin{defn}
   Given any pair $(\pi_I, \pi_J)$ (not necessarily incomparable), where   $I=[i_1, \cdots, i_\ell]$ and $J=[j_1, \cdots, j_m]$ with $\ell\geq m$,   the meet $\pi_{I\wedge J}$ and the join $\pi_{I\vee J}$ are defined respectively  by
  \begin{align*}
     I\wedge J&:=\{\min\{i_1,j_1\}, \cdots, \min\{i_m, j_m\}, i_{m+1}, \cdots, i_l\},\\
     I\vee J &:=\{\max\{i_1,j_1\}, \cdots, \max\{i_m, j_m\}\}.
  \end{align*}
\end{defn}
 Let   $X_0$ be the   subvariety of $\prod_{i=1}^k\mathbb{P}(\bigwedge^{n_i}\mathbb{C}^n)$
defined by the quadratic equations   $$p_Ip_J-p_{I\vee J}p_{I\wedge J}=0,$$
   one for each incomparable pair $(\pi_I, \pi_J)$; here $\{p_I:=p_{\pi_I}\}_I$ are naturally regarded as the multi-homogeneous coordinates of the product of projective spaces.
  Let $\lambda_{n+1}^{(n)}=0$ and set $b_j=\lambda^{(n)}_j-\lambda^{(n)}_{j+1}$ for each $1\leq j\leq n$. In particular, $b_j=0$ if $j\neq n_i$ for any $i$.
   We have
 \begin{prop}[Proposition 7 of \cite{KoMi}]\label{propEmbeddingsame}
    The  projective  embedding
    {\upshape $$X_0\hookrightarrow  \mathbb{P}\Big(\bigotimes_{i=1}^{k+1} \mbox{Sym}^{b_{n_i}}\big(\bigwedge^{n_i}\mathbb{C}^n\big)\Big)$$}
     is the projective embedding of the Gelfand-Cetlin  toric variety $X_{\Delta}$  associated to $\Delta$.
 \end{prop}
\noindent We remark that the above proposition was stated for   $(n_1, \cdots, n_k; n)=(1, 2, \cdots, n-1; n)$ in \cite{KoMi}, while it holds in general.
The coincidence  was obtained in \cite{KoMi} by making a bijection between the set $\Pi_\lambda$  of   lattice points of $\Delta$ and the set  of  global sections of  $\mathcal{O}_{X_0}(b_{n_1},\cdots, b_{n_{k+1}})$ (by which we mean the result of tensoring together with the pullbacks of the bundles $\mathcal{O}_{\mathbb{P}(\bigwedge^{n_i}\mathbb{C}^n )}(b_{n_i})$ to $X_0$).
For the later   study of toric subvarieties, we reinterpret the  identification in the current framework here.

Let us define an endomorphism of $\mathbb{Z}^{n(n+1)\over 2}$ by  $$\phi: \mathbb{Z}^{n(n+1)\over 2}\to \mathbb{Z}^{n(n+1)\over 2};  \{b_{ij}\}_{1\leq j\leq i\leq n}\mapsto \{\lambda^{(i)}_j:=b_{i,j}+b_{i-1, j}+\cdots+b_{j, j}\}.$$
In other words, if we  label every $\lambda^{(i)}_j$-box as a $b_{ij}$-box,    then the value $\lambda^{(i)}_j$ equals   the sum of values on the $b_{ij}$-box and all boxes below $b_{ij}$-box in column $j$.
Clearly, $\phi$ is a bijection with its inverse map   $\psi$     defined by   $b_{ij}=\lambda^{(i)}_j-\lambda^{(i-1)}_j$ (where $\lambda^{(i-1)}_i:=0$ for convention).
For each positive path $\pi_I$, we define a vector  $\beta_I=(\beta_{ij})\in \mathbb{Z}^{n(n+1)\over 2}$ by    setting $\beta_{ij}=1$ if the   $b_{ij}$-box is right above the path $\pi_I$, or $\beta_{ij}=0$ otherwise (see Figure \ref{PPforFl} for an illustration of the positive path $\tilde \pi=\tilde \pi_{1,2,5,7,8}$ for $F\ell_{3, 5; 8}$, where all the boxes assigned with $1$ are shown).\footnote{An exponent vector $\alpha_I\in \mathbb{Z}^{n^2}$ is equipped with $p_I$   in \cite{KoMi}.} Here for convention, we have   treated the path $\pi_{12\cdots n}$  (which starts at $O_0$, and moves along the bottom boundary of the square $Q$, towards the lower right corner of $Q$) as a positive path.
As in \cite{KoMi}\footnote{The case    $b_{n_1}= b_{n_2}=\cdots=b_{n_{k+1}}=1$  was proved in   \cite{BCFKS}.},
the vector space of global sections $\mathcal{O}_{X_0}(b_{n_1},\cdots, b_{n_{k+1}})$  decomposes into a multiplicity-one direct sum of weight spaces. Denote by ${[n]\choose j}$ the set of subsequences of $[1,2,\cdots, n]$ of cardinality $j$. The set of weights occurring in this decomposition is given by
 $$\Upsilon_\lambda=\big\{\sum_{I}\beta_{I}~|~ \mbox{there are exactly } b_j \mbox{ sequences } I\in {[n] \choose j} \mbox{ for } 1\leq j\leq n \big\}.$$
\begin{prop}
    The   bijection  $\phi$ sends $\Upsilon_\lambda$ onto  $\Pi_\lambda$. 
\end{prop}

 \noindent The statement follows from  exactly  the same arguments as on page 10 of \cite{KoMi} with the replacement of notation by $\lambda^{(i)}_j=\lambda_{n+1-i,j}$ and $b_{ij}=a_{n+1-i,j}$,   and hence yields Proposition \ref{propEmbeddingsame}.

\section{Toric degenerations   of $W$-translated Schubert varieties}
In this section, we     introduce the notion of $W$-translated Schubert varieties, and  study their toric degenerations.
\subsection{$W$-translated Schubert varieties
}
Let $G=SL(n, \mathbb{C})$ and  $B\subset G$ be the Borel subgroup of upper triangular matrices. Let $P\supset B$ be the parabolic subgroup of $G$ consisting of block-upper triangular  matrices with  diagonal blocks of the form   $\mbox{diag}(M_1,\cdots, M_{k+1})$  
         where    $M_i$ is an  $(n_i-n_{i-1})\times (n_i-n_{i-1})$ matrix.
   The (partial) flag variety   $X=G/P$  is smooth projective, parameterizing partial flags in $\mathbb{C}^n$:
    $$X=\{V_{n_1}\subseteq \cdots \subseteq V_{n_k}\subseteq \mathbb{C}^{n}~|~ \dim V_{n_i}=n_i, i=1, 2, \cdots, k\}=:F\ell_{n_1, \cdots, n_k; n}.$$

The Weyl group $W=S_n$ of $G$ is generated by the transpositions   $s_i=(i, i+1)$, $i=1,\cdots, n-1$.
 Let $T\subset G$   consist  of diagonal matrices,  and    $N(T)$ denote the normalizer of $T$ in $G$. There is a standard isomorphism of groups
   $$W\overset{\cong}{\longrightarrow} N(T)/T;\quad w=s_{i_1}\cdots s_{i_m}\mapsto  \dot s_{i_1}\cdots \dot s_{i_m}T.$$
   Here  $\dot s_i=\big(\sum_{j=1}^{i-1} E_{jj}\big)+E_{i, i+1}-E_{i+1,i}+\sum_{j=i+2}^n E_{jj}\in G$ where
   $E_{ij}$ denotes  the $n\times n$ matrix with the $(i, j)$-entry being 1, and 0 otherwise.
 {Although   $\dot w:=\dot s_{i_1}\cdots \dot s_{i_m}$ depends on the expression of $w=s_{i_1}\cdots s_{i_m}$,  the coset $\dot w T$ does not.}

 Let $\ell: W\to \mathbb{Z}_{\geq 0}$ be   the standard length function.  The flag variety   $X=G/P$ admits  a well-known Bruhat decomposition
      $$X=\bigsqcup_{w\in W^P} B\dot w P/P$$
where  $B\dot w P/P$ is isomorphic to the affine space $\mathbb{C}^{\ell(w)}$, and  $W^P$ is the subset    of $S_n$ given by
  $W^P=\{w\in S_{n}~|~ w(n_{i-1}+1)<w(n_{i-1}+2)<\cdots <w(n_i), i=1, \cdots, k+1\}.$
 The \textit{(opposite) Schubert varieties} $X_w$ (resp. $X^w$) of complex (co)dimension $\ell(w)$  are defined by
               $$X_w:=\overline{B\dot w P/P},\quad X^w:=\overline{\dot w_0 B\dot w_0\dot w P/P}  $$
where  
    $w_0$ is  the longest element in $W$.
 The  Schubert cohomology classes $\sigma^w:=P.D.[X^w]\in H^{2\ell(w)}(X, \mathbb{Z})$
form an  additive $\mathbb{Z}$-basis of $H^*(X, \mathbb{Z})$.



\begin{defn}
 A subvariety $Y\subset SL(n, \mathbb{C})/P$ is called a \textbf{$W$-translated Schubert variety}  if $Y=\dot uX_v$ for some $u\in W$ and $v\in W^P$.
\end{defn}

\noindent Clearly, the (opposite) Schubert varieties are special cases of the $W$-translated Schubert varieties. Moreover,
    $W$-translated Schubert varieties $\dot uX_v$ represent the same cohomology class as $X_v$, and are all $T$-invariant. Although the Weyl group  $W$ does not act on   $SL(n, \mathbb{C})/P$, it does act transitively on the set of $W$-translated Schubert varieties $\{\dot uX_v~|~ u\in W, v\in W^P\}$
 in the obvious way. We will simply denote $\dot uX_v$ as $uX_v$ by abuse of notation.

 The $W$-translated Schubert divisors are bijective to the Pl\"ucker coordinates, or equivalently, to the positive paths of the ladder diagram  $\Lambda(n_1, \cdots, n_k; n)$. More precisely,   
  there is a  well-known Pl\"ucker embedding 
  $$Pl: F\ell_{n_1, n_2, \cdots, n_k; n}\hookrightarrow   \mathbb{P}(\bigwedge^{n_1}\mathbb{C}^n)\times \cdots\times  \mathbb{P}(\bigwedge^{n_k}\mathbb{C}^n).$$
It can be  defined by the  \textit{straightening relations} (see \cite[section 9]{GoLa})
  $$p_Ip_J=p_{I\vee J}p_{I\wedge J}\pm\sum_{(I', J')}p_{I'}p_{J'},$$
  one for each incomparable pair $(\pi_I, \pi_J)$ (where each  $p_{I'}p_{J'}$ occurring in the sum satisfies
                 $ I'<I\wedge J< I\vee J<J'$).
  Given a positive path $\pi_I$, we define 
   $$D_{p_I}:=Pl^{-1}\big( \{p_I=0\}\cap Pl(X)\big),\quad\mbox{ or simply } D_{p_I}:=\{p_I=0\}\cap X$$
whenever there is no confusion.
As  a standard fact,  the   opposite  Schubert divisor $X^{s_{n_i}}$ is  defined by a single coordinate hyperplane   for each $i$, precisely given by
             $$X^{s_{n_i}}= D_{p_I} \mbox{ for } I=[1,2,3,\cdots, n_i].$$
We observe that    $W=N(T)/T$ acts transitively
  on  the set   $\{D_{p_J}\}_J$ via
          \begin{equation}
             w\cdot D_{p_J}= D_{p_{w(J)}},
          \end{equation}
 where the usual notation convention $p_{j_2j_1\cdots j_{n_\ell}}=-p_{j_1j_2\cdots j_{n_\ell}}$ for Pl\"ucker coordinates is adopted. 
 This yields the following.
\begin{prop}\label{propWSchubertdivisor} The set $\{uX^{s_{n_i}}~|~ u\in W, 1\leq i\leq k\}$ of    $W$-translated Schubert divisors coincides with the set $\{D_{p_I}\}_I$, precisely   given by
   $$uX^{s_{n_i}}=D_{p_{u([1,2,\cdots, n_i])}}.$$
  \end{prop}

 Every Schubert variety $X_w$ is scheme theoretically 
the intersection of $X$ with all the Pl\"ucker coordinate hyperplanes containing $X_w$ (see e.g.   \cite[section 2.10]{BiLa}). Consequently, every $W$-translated Schubert variety is scheme theoretically 
the intersection of  $X$ with all the $W$-translated Schubert divisors containing it.
Moreover, the $W$-translated Schubert divisors behave very well with respect to the toric degeneration of $X$, as we will see  in the next subsection.

\subsection{Toric degenerations  of $W$-translated Schubert divisors for $F\ell_{n_1,\cdots, n_k;n}$}

There have been lots of studies of the toric degenerations of the partial flag variety $X$ to the Gelfand-Cetlin toric variety $X_{\Delta}$  \cite{GoLa,BCFKS,KoMi,KnMi,NNU}.
Although the degenerations are realized in different ways, they satisfy the following common   property: the degeneration is  a flat family
$$  \xymatrix{
  \mathcal{X} \ar[d]_{\varrho}  \ar[r]^{\subset\qquad\qquad\qquad\quad{}}  & \ar[dl]_{}       \mathbb{P}(\bigwedge^{n_1}\mathbb{C}^n)\times \cdots\times  \mathbb{P}(\bigwedge^{n_k}\mathbb{C}^n)\times \mathbb{C}     \\
  \mathbb{C}                 }
$$
such that   $\mathcal{X}|_{\mathbb{C}^*}\cong X\times \mathbb{C}^*$, and  $\mathcal{X}|_{t=0}= {X_{\Delta}}\times\{0\}$.
It is then natural to ask whether a flat subfamily exists and  how it looks like for a   reasonably nice subvariety of $X$.
In the case of complete flag variety $F\ell_{1,2, \cdots,n-1; n}$, this has been well studied by Kogan and Miller \cite{KoMi} for (opposite) Schubert varieties, which are special $W$-translated Schubert varieties. They  degenerate to reduced union of  the toric subvarieties of so-called (dual) Kogan faces of the Gelfand-Celtin polytope $\Delta$.
Here we will restrict us to divisors, but will study the general $W$-Schubert divisors of partial flag varieties $F\ell_{n_1,\cdots, n_k; n}$.

Let us use the degeneration given by  Gonciulea and  Lakshmibai \cite{GoLa}, which is defined by equations of the form
 $$p_Ip_J=p_{I\vee J}p_{I\wedge J}\pm\sum_{(I', J')}t^{N_{I'}+N_{J'}-N_{I}-N_{J}}p_{I'}p_{J'},$$
   one for each incomparable pair $(\pi_I, \pi_J)$.
The isomorphism $\mathcal{X}|_{t=1} \overset{\cong}{\longrightarrow}\mathcal{X}|_{t}$ (where $t\neq 0$) is simply given by  $p_I\mapsto t^{N_I}p_I$, where the exponents $N_I=N_{[i_1, \cdots, i_\ell]}$ can be defined for instance  by the number
  $\sum_{r=1}^\ell (2n)^{n-r}i_r.$
It follows immediately that any closed subvariety $\mathcal{Z}\subset \mathcal{X}$ given by the intersection of Pl\"ucker coordinate hyperplanes is flat over $\mathbb{C}^*$. Such $\mathcal{Z}|_{\mathbb{C}^*}$ admits a flat extension to the central fiber due to the next proposition.

\begin{prop}[Proposition III. 9.8 of \cite{Hart}]\label{propHart}
  Let $Y$ be a regular, integral scheme of dimension 1, let $p\in Y$ be a closed point, and let $\mathcal{Z}^*\subseteq \mathbb{P}^N_{Y-p}$ be a closed subscheme which is flat over $Y-p$. Then there exists a unique closed subscheme $\overline{\mathcal{Z}^*}\subset \mathbb{P}^N_Y$, flat over $Y$, whose restriction to $\mathbb{P}^N_{Y-p}$ is $\mathcal{Z}^*$.
\end{prop}

 Now we consider the case when the subvariety of $X$ is   a $W$-translated Schubert divisor. Recall that every effective edge $e$ in the ladder diagram
  $\Lambda=\Lambda(n_1, \cdots, n_k; n)$ defines a facet $F_e$ of the Gelfand-Cetlin polytope $\Delta$. We have denoted by   $X_{F_e}$
   the toric divisor of the Gelfand-Celtin toric variety $X_\Delta$ associated to the facet $F_e$.

 \begin{prop}\label{proptoricdivisor} For any  effective edge $e$ of  $\Lambda$, the corresponding toric divisor $X_{F_e}$, as a subvariety of $X_0$ in  $\prod_{i=1}^k\mathbb{P}(\bigwedge^{n_i}\mathbb{C}^n)$,  is given by
  $$X_{F_e}=X_0\bigcap\{ p_I=0~|~  \pi_I \mbox{ is a positive path containing } e\}.$$
 \end{prop}

\begin{proof}
It is sufficient to show the coincidence of both sides under the embedding in Proposition \ref{propEmbeddingsame}. The set of weights of global sections of $\mathcal{O}_{X_{F_e}}(b_{n_1},\cdots, b_{n_{k+1}})$ is the subset $\Upsilon_\lambda^e\subset \Upsilon_\lambda$ that satisfies $\phi(\Upsilon_\lambda^e)\subset F_e$. We claim
  $$\Upsilon_\lambda^e=\big\{\sum_I\beta_I~|~\mbox{there are } b_j \mbox{ sequences  } I\in{[n]\choose j}  \mbox{ for all } j;   \mbox{ none of  }  \pi_I \mbox{ contain } e\big\}.
 $$
Assuming this claim first, we denote by $D(e)$ the subvariety of $X_0$ defined by the equations on the right hand side in the statement.
Clearly, $D(e)$ is a torus-invariant closed subvariety of $X_0$ with respect to the maximal torus action on $X_0$, and $\Upsilon_\lambda^e$ is a subset of the  set of weights of global sections of $\mathcal{O}_{D(e)}(b_{n_1},\cdots, b_{n_{k+1}})$. Hence, $X_{F_e}\subset D(e)$ (though set-theoretically, a priori).   On the other hand, given an incomparable pair $(\pi_I, \pi_J)$, we observe that either of the positive paths $\pi_I, \pi_J$ contains edge $e$ if and only if so does   either of the positive paths $ \pi_{I\vee J}, \pi_{I\wedge J}$.  It follows that $\{\pi_I~|~ \pi_I \mbox{ does not contain } e\}$ is a finite distributive sublattice of $\{\pi_I\}_I$, and that $D(e)$ is the (reduced, irreducible) toric variety associated to such a distributive sublattice (see for instance   \cite[Theorem 4.3]{GoLa} for the standard fact about   toric varieties associated to distributive lattices). Clearly, $D(e)\subsetneq X_0$. Hence,  $X_{F_e}= D(e)$ as varieties.

It remains to prove the claim.
Indeed, if $e$ is a horizontal (resp. vertical) effective edge of $\Lambda$, then it is the common edge of $\lambda^{(i)}_j$-box and $\lambda^{(i-1)}_j$-box (resp. $\lambda^{(i+1)}_{j+1}$-box) for some $i, j$.
 We notice that  a positive path  $\pi_I$  does not contain $e$ if and only if  $\pi_I$ is   either above $\lambda^{(i)}_j$-box or below  $\lambda^{(i-1)}_j$-box (resp.  either above or below both of the two boxes).
Take  $\sum_I\beta_I\in \Upsilon_\lambda^e$. If $e$ is horizontal (resp. vertical),  then for every such  $\beta_I$, the $b_{ij}$-box always has value 0 (resp. the sum of the values of the boxes in the column between the $b_{ij}$-box and the bottom is equal to that between the $b_{i+1, j+1}$-box and the bottom).   It follows from the definition of $\phi$ that $\lambda^{(i)}_j=\lambda^{(i-1)}_j$ (resp. $\lambda^{(i)}_j=\lambda^{(i+1)}_{j+1}$). Hence, $\phi(\Upsilon_\lambda^e)\subset F_e$.
 Conversely, given $\chi\in F_e$, we  write $\psi(\chi)=\sum_I\beta_I$. With the same arguments above, we conclude that none of the  positive  paths $\pi_I$ contain $e$, so that $\psi(F_e)\subset \Upsilon^e_\lambda$.
 \end{proof}

\begin{thm}\label{thmToricDeg} For any positive path $\pi_I$, the subvariety  $\mathcal{D}_{p_I}:= \{p_I=0\}\cap \mathcal{X}$ of $\mathcal{X}$ is a flat subfamily, whose generic fiber
  is isomorphic to the   $W$-translated Schubert divisor $D_{p_I}$ of the partial  flag variety $X$, and whose fiber at $t=0$ is a  (reduced)   union of toric divisors of the Gelfand-Cetlin toric variety $X_0=X_\Delta$:
    {\upshape  $$\bigsqcup_{e\subset \pi_I \scriptsize\mbox{ is effective}}X_{F_e}.$$
}
\end{thm}
\begin{proof}It follows immediately from the aforementioned realization of $\mathcal{X}$ that $\mathcal{D}_{p_I}|_{t\neq 0}$ is flat over $\mathbb{C}^*$. (In fact, more is true:
     $p_I\mapsto t^{N_I}p_I$ induces an isomorphism $\mathcal{D}_{p_I}|_{t\neq 0}\cong D_{p_I} \times \mathbb{C}^*$.) We notice that $\mathcal{D}_{p_I}|_{t=0}=X_{\Delta}\cap \{p_I=0\}$ is a reduced torus-invariant divisor of $X_0$. In particular, it is of pure dimension $\dim X_0-1=\dim D_{p_I}$. Hence, $ \mathcal{D}_{p_I}$ is irreducible. (Otherwise, $\mathcal{D}_{p_I}$ would have a component, which is a hypersurface in $\mathcal{X}$ and whose image under $\varrho$ is a point. Such a component of dimension $(1+\dim D_{p_I})$ can only occur in $t=0$, resulting in a contradiction.)
Since  $\mathcal{D}_{p_I}$ is closed, it follows that the unique flat extension $\overline{\mathcal{D}_{p_I}|_{t\neq 0}}$ by Proposition \ref{propHart} coincides with $\mathcal{D}_{p_I}$. The description of the fiber of $\mathcal{D}_{p_I}$ at $t=0$ follows from the explicit defining equations of toric divisors as in Proposition \ref{proptoricdivisor}.
\end{proof}

\begin{remark}\label{geomexp}
 Apply  Theorem \ref{thmToricDeg} to the case $X=F\ell_n$ with $\pi_I=\pi_{n+1-i_0, n+2-i_0, \cdots, n}$, where $1\leq i_0\leq n-1$. Then   the facets of $\Delta$ occurring in the toric degeneration of the $W$-translated Schubert divisor $D_{p_I}=uX^{s_{i_0}}$ are precisely given by $\{\Gamma_{i_0-j,j}\mid 1\leq j\leq i_0\}$ in terms of the notation in \cite{KST}. In other words, it is consistent with  the representation of the divisor class
 $[uX^{s_{i_0}}]=\sigma^{s_{i_0}}$ by elements of polytope ring in \cite[Example 4.4]{KST}. Moreover, by induction and \cite[Proposition 3.2]{KST}, one can show that the following two collections of sets of facets coincide with each other:
  $$\{\{F_e\}_e\mid  \mbox{there exists   }   uX^{s_i} \mbox{ that degenerates to  }\bigsqcup_{e}X_{F_e} \};$$
    $$\{\{F_e\}_e\mid  \mbox{there exists   }   X^{s_i} \mbox{such that }\sum_e\pi([F_e])=\sigma^{s_i} \mbox{ in the sense of \cite[section 4.3]{KST}}\}.$$
\end{remark}
\subsection{Toric degeneration of $W$-translated Schubert varieties for $Gr(m, n)$}\label{subsecDefEqnGrassman}  Toric degeneration of  Schubert varieties of complex Grassmannians $Gr(m, n)=F\ell_{m; n}$ have been studied well.
Although it is known to the experts, a precise description of such degenerations in terms of faces of the Gelfand-Cetlin polytopes seems  to be missing in the literature. In this subsection, we will give precise descriptions for the (opposite) Schubert varieties of $Gr(m, n)$ as well as for certain $W$-translated Schubert varieties of $Gr(2,n)$.

\subsubsection{Indexing set by partitions}Schubert varieties of $Gr(m, n)$ are traditionally indexed by   the set of partitions  inside an $m\times (n-m)$ rectangle,
  $$\mathcal{P}_{m, n}:=\{\mu=(\mu_1, \cdots, \mu_m)\in \mathbb{Z}^{m}~|~ n-m\geq \mu_1\geq \mu_2\geq \cdots\geq \mu_m\geq 0\}, $$
which is equipped with a natural partial order $\leq$ as for vectors in $\mathbb{Z}^m$. There is an isomorphism of partially ordered sets
 $\big(\{\pi_I\}_I, \leq )\overset{\cong}{\longrightarrow} (\mathcal{P}_{m, n}, \leq)$ defined by
    $$ p_{i_1i_2\dots i_{m}}\mapsto \mu=(i_m-m, \cdots, i_2-2, i_1-1).$$
The partition $\mu$ can be directly read off from the ladder diagram by counting the number of boxes below the positive path  column-to-column from right to left. The complement $\mu^\vee:=(n-m-\mu_m, \cdots, n-m-\mu_1)$ can also be directly read off from the ladder diagram by counting the number of boxes above the positive  path column-to-column from left to right. 
  \begin{figure}[h]
  \caption{Partitions and faces associated to a positive path   in $\Lambda(5; 8)$}\label{figurepartitions}
  \bigskip
     \includegraphics[scale=1]{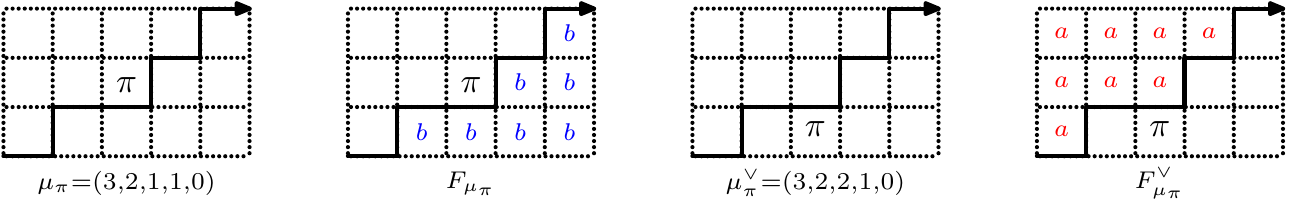}
  \end{figure}

\subsubsection{Defining equations of    toric  subvarieties for $\Lambda(m; n)$}
 Given  a partition $\mu=\mu_\pi$ corresponding to a positive path $\pi$, we define a face $F_{\mu_\pi}$ (resp. $F_{\mu_\pi}^\vee$) of $\Delta$ simply by letting  all the boxes below (resp. above) the path $\pi$ take the same value     $b:=\lambda^{(n)}_n$ (resp. $a:=\lambda^{(n)}_m$).  (See Figure \ref{figurepartitions} for an illustration.)
 By similar  arguments   to the proof of  Proposition \ref{proptoricdivisor}, we have

\begin{prop}\label{propDefEqnTorforSchuGrass}
The toric subvarieties  $X_{F_{\mu_\pi}}$ and   $X_{F_{\mu_\pi}^\vee}$, as subvarieties  of $X_0$ in  $\mathbb{P}(\bigwedge^{m}\mathbb{C}^n)$, are respectively    given by
  $$X_{F_{\mu_\pi}}=X_0\bigcap\{ p_I=0~|~  \pi_I\not\geq \pi\},\quad X_{F_{\mu_\pi}^\vee}=X_0\bigcap\{ p_I=0~|~  \pi_I\not\leq \pi\}.$$
\end{prop}

\begin{proof}
We notice that   $ \pi_I\not\geq \pi$ if and only if either  $\pi_I\leq \pi$ or the pair $(\pi_I, \pi_\pi)$ is incomparable.
Under the  embedding
    $X_0\hookrightarrow  \mathbb{P}\big(\mbox{Sym}^{a-b}(\bigwedge^{m}\mathbb{C}^n)\bigotimes \mbox{Sym}^{b}(\bigwedge^{n}\mathbb{C}^n)\big)$
as in Proposition \ref{propEmbeddingsame}, the set  $\Upsilon_\lambda$ of weights of global sections is given by
 $$\Upsilon_\lambda=\big\{ b\beta_{12\cdots n}+\sum_{j=1}^{a-b}\beta_{I_j}~|~  \pi_{I_j} \mbox{ is a positive path with } m \mbox{ horizontal steps for each } j\big\}.$$
Moreover, $\Upsilon_\lambda$ is sent onto the lattice points of the Gelfand-Cetlin polytope $\Delta$ 
 via the bijection $\phi$.
Consider the subsets
\begin{align*}
    \Upsilon_\lambda^{\mu_\pi}&:=\big\{ b\beta_{12\cdots n}+\sum_{j=1}^{a-b}\beta_{I_j}~|~  \pi_{I_j}\geq \pi  \mbox{ for each } j\big\}\cap  \Upsilon_\lambda\subset \Upsilon_\lambda,\\
    \Upsilon_\lambda^{\mu_\pi^\vee}&:=\big\{ b\beta_{12\cdots n}+\sum_{j=1}^{a-b}\beta_{I_j}~|~  \pi_{I_j}\leq \pi  \mbox{ for each } j\big\}\cap \Upsilon_\lambda\subset \Upsilon_\lambda.
     \end{align*}
The vector $\beta_{12\cdots n}$ takes value $1$ on the entries labeled by the boxes on the bottom row of the square $Q$, or value $0$ otherwise.
For any $\beta_{I_j}$ in an element $x$ of   $\Upsilon_\lambda^{\mu_\pi}$ (resp.   $\Upsilon_\lambda^{\mu_\pi^\vee}$), the nonzero entries of the vector $\beta_{I_j}$ occur only on some boxes right above the positive path $\pi_{I_j}$, and hence above $\pi$ (resp. and hence make contributions 1 for any box above $\pi$ in the image   $\phi(x)$). It follows immediately that $\phi(\Upsilon_\lambda^{\mu_\pi})\subset F_{\mu_\pi}$ and   $\phi(\Upsilon_\lambda^{\mu_\pi^\vee})\subset F_{\mu_\pi}^\vee$.
Now take any element $y$ of $F_{\mu_\pi}$ (resp. $F_{\mu_\pi}^\vee$). If $\pi_{I_j}\geq \pi$ (resp. $\pi_{I_j}\leq \pi$) did not hold for some $\beta_{I_j}$ in $\psi(y)$, then there must exist a box having a top (resp. bottom) edge on the path $\pi$ and being above (resp. below) $\pi_{I_j}$. It would make contribution $1$ (resp. $0$)
to such box in the image $\phi(\psi(y))$. Consequently, this box would take value larger than $b$ (resp. smaller than $a$) in $y=\phi(\psi(y))$, resulting a contradiction.
Hence, there is a bijection between  $\Upsilon_\lambda^{\mu_\pi}$ (resp. $\Upsilon_\lambda^{\mu_\pi^\vee}$) and the lattice points in $F_{\mu_\pi}$ (resp. $F_{\mu_\pi}^\vee$) via the map $\phi$.

On the other hand, we denote $X_0(F_{\mu_\pi}):= X_0\bigcap\{ p_I=0~|~  \pi_I\not\geq \pi\}$ and $X_0(F_{\mu_\pi}^\vee):= X_{F_{\mu_\pi}^\vee}=X_0\bigcap\{ p_I=0~|~  \pi_I\not\leq \pi\}$.
Clearly, they are both torus-invariant, and the sets of  weights of the global sections of the restriction of $\mathcal{O}_{X_0}(a-b, b)$ to them contain $\Upsilon_\lambda^{\mu_\pi}$ and
    $\Upsilon_\lambda^{\mu_\pi^\vee}$ respectively. Furthermore, they are the  toric varieties respectively associated to the finite distributive sublattices
 $\{\pi_I~|~ \pi_I\geq {\pi}\} $  and  $\{\pi_I~|~ \pi_I\leq {\pi}\}$. It follows that
     $X_{F_{\mu_\pi}}\subset X_0(F_{\mu_\pi})$ and $X_{F_{\mu_\pi}^\vee}\subset X_0(F_{\mu_\pi}^\vee)$ as reduced, irreducible varieties. The ranks of these  lattices are respectively given by
 \begin{align*}\max\{r-1~|~ \pi_{I_1}>\pi_{I_2}>\cdots>\pi_{I_r}\geq  \pi\}&=\sum_{i=1}^m (n-m-\mu_i)
 =\dim_{\mathbb{R}} F_{\mu_\pi}, \\
 \max\{r-1~|~ \pi_{I_1}<\pi_{I_2}<\cdots<\pi_{I_r}\leq  \pi\}&=\sum_{i=1}^m \mu_i=|\mu_\pi|=\dim_{\mathbb{R}} F_{\mu_\pi}^\vee.
 \end{align*}
It is well-known that the rank of a finite distributive lattice is equal to the   complex dimension  of the associated projective toric variety (see e.g.  \cite[section 1]{Hibi}).
Therefore, $X_0(F_{\mu_\pi}^\vee)$ (resp. $X_0(F_{\mu_\pi})$)  is of (co)dimension $|\mu|$, same as that 
of $X_{F_{\mu_\pi}^\vee}$ (resp. $X_{F_{\mu_\pi}}$). Hence, they must coincide with each other.
\end{proof}

\subsubsection{Toric degeneration of Schubert varieties}
The (opposite) Schubert varieties of $X=Gr(m, n)=SL(n, \mathbb{C})/P$ are indexed by the subset $W^P=\{w\in S_n~|~ w(1)<\cdots<w(m); \quad w(m+1)<\cdots w(n)\}$ of permutations. There are bijections of sets
 $$ W^P\overset{\cong}{\longrightarrow} \{p_\pi\}_{\pi} \overset{\cong}{\longrightarrow}\mathcal{P}_{m, n};\quad w\mapsto p_{w(1), \cdots, w(m)} \mapsto \mu_w=(w(m)-m, \cdots, w(1)-1).$$
Traditionally,  the (opposite) Schubert varieties  $X_{\mu_w}:=X_w$,    $X^{\mu_w}:=X^w$ are indexed by the partitions   $\mu_w\in\mathcal{P}_{m, n}$. We simply denote $\mu=\mu_w$.
The defining equations of (opposite) Schubert varieties, as subvarieties of $\mathbb{P}(\bigwedge^m\mathbb{C}^n)$ via the Pl\"ucker embedding, are just given by
(see e.g. Section 4.3.4 of \cite{LaRa})
  $$X_\mu=X\cap \{p_I=0~|~ \pi_I\not\leq  {\pi_\mu}\},\qquad X^\mu=X\cap \{p_I=0~|~ \pi_I\not\geq {\pi_\mu}\}. $$
 \begin{prop}\label{propTorDegSchuGrass} The subfamilies $\mathcal{X}_\mu:=\mathcal{X}\cap \{p_I=0~|~ \pi_I\not\leq {\pi_\mu}\}$ and $\mathcal{X}^\mu:=\mathcal{X}\cap \{p_I=0~|~ \pi_I\not\geq {\pi_\mu}\}$ are flat over $\mathbb{C}$. Moreover,
   $$\mathcal{X}_\mu |_{t=1}=X_\mu, \quad \mathcal{X}_\mu |_{t=0}= X_{F_{\mu}^\vee},\quad
     \mathcal{X}^\mu |_{t=1}=X^\mu, \quad \mathcal{X}^\mu |_{t=0}= X_{F_{\mu}}. $$
   Furthermore for any partition $\eta\in \mathcal{P}_{m, n}$ with $\eta\geq \mu$, the subfamily $\mathcal{X}^\mu\cap \mathcal{X}_\eta$ is   flat.
\end{prop}
\begin{proof}
 We notice that  $\mathcal{X}_\mu|_{\mathbb{C}^*}$ is flat, isomorphic to $X_\mu\times \mathbb{C}^*$. Hence, it   extends to flat subfamily $\overline{\mathcal{X}_\mu|_{\mathbb{C}^*}}\subset \mathcal{X}_\mu$ with $\overline{\mathcal{X}_\mu|_{\mathbb{C}^*}}|_{t=0}\subset \mathcal{X}_\mu|_{t=0}$. By Proposition \ref{propDefEqnTorforSchuGrass}, $\mathcal{X}_\mu|_{t=0}= X_{F_{\mu}^\vee}$. In particular, it is reduced, irreducible, and of dimension $|\mu|=\dim X_\mu$. Since
    $\dim \overline{\mathcal{X}_\mu|_{\mathbb{C}^*}}|_{t=0}=   \dim \overline{\mathcal{X}_\mu|_{\mathbb{C}^*}}|_{t=1}=|\mu|$,  it follows that
     $\overline{\mathcal{X}_\mu|_{\mathbb{C}^*}}|_{t=0}=X_{F_{\mu}^\vee}$ and  $\overline{\mathcal{X}_\mu|_{\mathbb{C}^*}}=\mathcal{X}_\mu$.
      The arguments for $\mathcal{X}^\mu$ are the same.

      For  $\eta\geq \mu$, $X_{F_{\eta}^\vee}\cap X_{F_\mu}=X_{F_{\eta}^\vee\cap F_\mu}$ is a (reduced, irreducible) toric subvariety of dimension $|\eta|-|\mu|=\dim X^\mu\cap X_\eta$. Thus the last statement follows as well.
\end{proof}

\subsubsection{Toric degeneration of special $W$-translated  Schubert varieties in $Gr(2, n)$}
Let $C$ denote the cyclic permutation $(2,3,\cdots, n, 1)$. Here we study the toric degeneration of the $W$-translated Schubert varieties $C^k\cdot X^{(1,1)}$ of $Gr(2, n)$.
We remark that the permutation $C$ plays an interesting role in the study of mirror symmetry for complex Grassmannians \cite{Karp}. Note that $C$ is of order $n$, and that $C^0 \cdot X^{(1,1)}=X^{(1,1)}$ and $C^{n-1}\cdot X^{(1,1)}=X_{(n-3,n-3)}$ have been studied above.
The Gelfand-Cetlin polytope of $Gr(2, n)$ are parameterized by variables $\{\lambda_1^{(i)},\lambda_2^{(i+1)}\mid 1\leq i\leq n-2\}$.
Let $\Delta_{(k)}$ denote the face of $\Delta$ of codimension two defined by
$$\Delta_{(k)}=\Delta\cap\{\lambda_1^{(k)}=\lambda_1^{(k+1)},
\lambda_{2}^{(k)}=\lambda_2^{(k+1)}\},$$
where we denote $\lambda_2^{(1)}=b$ and $\lambda_1^{(n-1)}=a$ for conventions.
\begin{prop}\label{proptorforGr2n}
  For $1\leq k\leq n-2$,  the $W$-translated Schubert variety $C^k\cdot X^{(1,1)}$ of $X=Gr(2, n)$ degenerates to the toric subvariety of $X_0$ defined by the face $\Delta_{(k)}$.
\end{prop}

\begin{proof}
   The opposite Schubert variety $X^{(1,1)}$ is given by the scheme-theoretical intersection $X\cap \{p_{1j}=0\mid 2\leq j\leq n\}$.
   Therefore  $C^{k}\cdot X^{(1,1)}$ is given by the scheme-theoretical intersection $X\cap \{p_{k+1,j}=0\mid 1\leq j\leq n, j\neq k+1\}$, and it degenerates to a (codimension two) closed subvariety of $X_0\cap \{p_{k+1,j}=0\mid 1\leq j\leq n, j\neq k+1\}$. By direct verifications, $\Delta_{(k)}$ is the set-theoretically given by   $\bigcap_{1\leq j\leq n, j\neq k+1}\bigcup_{e\subset p_{k+1, j}}F_e$. Hence, $C^{k}\cdot X^{(1,1)}$ degenerates to a priori a multiplicative copy of the (reduced, irreducible) toric subvariety defined by the face $\Delta_{(k)}$. On the other hand, the coordinate point by all coordinate hyperplanes of $\mathbb{P}^{{n \choose 2}-1}$ but one single $p_{ij}$ with $k+1\not\in\{i, j\}$ is obviously a reduced point in the scheme-theoretical intersection of
    both $C^k\cdot X^{(1,1)}$ and $X_0\cap \{p_{k+1,j}=0\mid 1\leq j\leq n, j\neq k+1\}$. It follows that the multiplicity is equal to one. That is, the statement follows.
   \end{proof}
\noindent   We remark that $C\cdot X^{(r, t)}$ degenerates to the union of more than one toric subvarieties of $X_0$ for general $n-2\geq r\geq t\geq 0$. For instance the subvariety $C\cdot X^{(2, 0)}$ of $Gr(2, 5)$ degenerates to three distinct toric subvarieties of $X_0\subset \mathbb{P}^9$, each of which is  isomorphic to $\mathbb{P}^4$.
\section{Applications: transversal intersections of Schubert varieties}
In this section,  we will study the transversality of intersections of $W$-translated Schubert varieties  by using  toric degenerations. We will  restate  Conjecture
 \ref{mainconj00} in Conjecture \ref{mainconj33}, which  leads to a  Littlewood-Richardson rule for complete flag varieties.  We will verify our conjecture in some cases.

\subsection{Transversal intersections by toric degenerations} The natural projection from $F\ell_n:=F\ell_{1,2,\cdots, n-1; n}$ to
$F\ell_{n_1, \cdots, n_k; n}$ induces an injective homomorphism of algebras $H^*(F\ell_{n_1, \cdots, n_k; n}, \mathbb{Z})\hookrightarrow H^*(F\ell_n,\mathbb{Z})$, which sends every Schubert class $\sigma^u_{P}$ of the former one to the Schubert class $\sigma^u_B$   of the latter one. Therefore, we will skip the subscripts $P, B$ whenever it is well understood, and will use the same notation for the corresponding Schubert structure constants. The Schubert classes $\{\sigma^u\}_{u\in S_n}$ form an additive $\mathbb{Z}$-basis of $H^*(F\ell_n,\mathbb{Z})$, so that for any $u_1,\cdots, u_m, v\in S_n$, we can write
$$\sigma^{u_1}\cup \cdots \cup \sigma^{u_m}\cup \sigma^{v}=\sum_{w\in S_n}N_{u_1,\cdots, u_m, v}^w\sigma^w.$$
We usually consider the case $m=1$ only, since it determines   the cases $m\geq 2$. The structure constant
$N_{u_1,\cdots, u_m, v}^w$ is non-negative, counting the number of intersection points of $g_1\cdot X^{u_1}\cap\cdots \cap g_m\cdot X^{u_m}\cap g_{m+1}\cdot X^v\cap X_w$ for generic $(g_1,\cdots, g_{m+1})\in \prod_{i=1}^{m+1} SL(n, \mathbb{C})$ if $\ell(w)=\ell(u_1)+\cdots+\ell(u_m)+\ell(v)$, or zero otherwise. One of the central problems in Schubert calculus is finding a Littlewood-Richardson rule, namely a manifestly combinatorial formula/algorithm for the Schubert structure constants. However, this problem is widely open except for complex Grassmannians and for two-step flag varieties \cite{Coskun-twostep,Buch-equivTwostep,BKPT}.

We will make the statement valid for any partial flag variety $X=SL(n,\mathbb{C})/P=F\ell_{n_1,\cdots, n_k; n}$.
 Denote by $V$ the set of vertices of  $\Delta$ for $X$, and let
 $$
   V^{X}:=\{z\in V\mid X_0(z)\in X\},$$
where    $X_0(z)$ denotes the   $0$-dimensional toric subvariety of $X_0$  corresponding to $z$; here we consider $X_0(z)$ as an element of the product of projective spaces that contains both $X_0$ and $X$.
We call   $[z_0, z_1,\cdots, z_N]$ in $\mathbb{P}^N$ a coordinate point if all $z_j$ are equal to $0$ but one.
\begin{lemma}\label{lemverticesGr}
  Let $X=Gr(m, n)$. Then  $\{X_0(z)\mid z\in V\}$   is the set  of the coordinate points of $\mathbb{P}^{{n\choose m}-1}$. In particular, $V^X=V$.
\end{lemma}
\begin{proof}
  Notice that a vertex  in the polytope $\Delta$ is an assignment of the boxes by either $a=\lambda_m^{(n)}$ or $b=\lambda_n^{(n)}$.
  Since  $\lambda^{(i+1)}_j\geq \lambda_j^{(i)}\geq \lambda^{(i+1)}_{j+1}$, for any $2\leq j\leq m$, we have the inequality
   $0\leq \sharp\{i\mid \lambda_{j-1}^{(i)}=b\}\leq \sharp\{i\mid \lambda_{j}^{(i)}=b\}\leq n-m$. In other words, it determines a unique partition and hence the corresponding positive path $\pi_I$. By Proposition \ref{proptoricdivisor}, the corresponding point in $X_0$ does not lie in the coordinate plane $\{p_I=0\}$ but lies in any other coordinate plane, and hence is a coordinate point. The number of positive paths is equal to ${n\choose m}$, which also equals the number of coordinate points of  $\mathbb{P}^{{n\choose m}-1}$.
   Therefore the correspondence is one-to-one, and hence the first statement follows. The second statement also follows immediately, by noting that every coordinate point satisfies all the defining equations of $X$, which are all sum of the form $p_Jp_K$ with $J\neq K$ up to a sign.
\end{proof}
\begin{cor}\label{coreqnvertices}
   Let $X=F\ell_{n_1,\cdots, n_k; n}$. For any $z\in V$, $X_0(z)$ is a coordinate point of $\mathbb{P}^{{n\choose n_1}-1}\times\cdots\times \mathbb{P}^{{n\choose n_k}-1}$.
\end{cor}
\begin{proof}
 The restriction of the natural projection   $\mathbb{P}^{{n\choose n_1}-1}\times\cdots\times \mathbb{P}^{{n\choose n_k}-1}\to\mathbb{P}^{{n\choose n_j}-1}$ to the Gelfand-Cetlin toric variety $X_0$ has image in the Gelfand-Cetlin toric variety with respect to $\Lambda(n_j;n)$. Moreover, it sends $X_0(z)$ to a point of the latter toric variety that corresponds to a vertex of the latter polytope. Therefore the statement follows from Lemma \ref{lemverticesGr}.
\end{proof}
\begin{prop}
 If $X=F\ell_n$, then $V^X$ consists of vertices of $\Delta$ such that for any $1\leq j\leq i<i+2\leq n$,   not all the equalities  $\lambda_j^{(i)}=\lambda_{j+1}^{(i+1)}=\lambda_{j}^{(i+1)}=\lambda_{j+1}^{(i+2)}$ hold.
\end{prop}
\begin{proof}
   Let   $z=\{\lambda_{c}^{(d)}\}$ be a vertex in $\Delta$.  For $1\leq r\leq n$, let us simply call $\{\lambda_{c}^{(d)}\mid \lambda_{c}^{(d)}=\lambda^{(n)}_r\}$ the $\lambda^{(n)}_r$-block. The inequalities for the Gelfand-Celtin patterns ensure that the $\lambda_r^{(n)}$-block and the $\lambda_{r+1}^{(n)}$-block are adjacent by a unique positive path $\pi_{I_r}$ with $|I_r|=r$ for any $1\leq r\leq n-1$. These positive paths defines a coordinate point of the product of projective spaces.

 For the given vertex $z$, we notice that equalities of the form  $\lambda_j^{(i)}=\lambda_{j+1}^{(i+1)}=\lambda_{j}^{(i+1)}=\lambda_{j+1}^{(i+2)}$ can at most occur in the  $\lambda^{(n)}_r$-block for some $2\leq r\leq n-1$. If they do occur in some $\lambda^{(n)}_r$-block,  we can assume $r=2$ without loss of generality. Let $k_j=\min\{i\mid \lambda_j^{(i)}=\lambda_2^{(n)}\}$ for $j\in \{1, 2\}$, and let $k_3=\max\{i+1\mid \lambda_1^{(i)}=\lambda_2^{(n)}\}$. It follows that $k_1<k_2<k_3$.    In this case, the coordinates of $X_0(z)$ satisfy  $p_{k_3}p_{k_1k_2}\neq 0$,
 $p_{k_1}p_{k_2k_3}=0=p_{k_2}p_{k_1k_3}$. Since one defining equation of $F\ell_n$ is given by $p_{k_1}p_{k_2k_3}-p_{k_2}p_{k_1k_3}+p_{k_3}p_{k_1k_2}=0$, $X_0(z)$ cannot belong to $F\ell_n$.

 Now we consider the case for any $i$, not all the equalities  $\lambda_1^{(i)}=\lambda_{2}^{(i+1)}=\lambda_{1}^{(i+1)}=\lambda_{2}^{(i+2)}$ hold for the given vertex $z$.
 The nonvanishing coordinates $p_{I_1}$ and $p_{I_2}$ of $X_0(z)$, where $|I_i|=i$ for $i\in \{1, 2\}$, satisfy either (a) $I_1=\{\hat k\}$, $I_2=\{\hat k, k\}$ or (b) $I_i=\{k\}$, $I_2=\{\hat k, k\}$, for some  $\hat k<k$. In particular, $p_{I_1}p_{I_2}$ never occurs in the defining equations with respect to incomparable pairs $(\pi_J, \pi_K)$ with $|J|=|I_1|$ and $|K|=|I_2|$. In other words, $X_0(z)$ satisfy all the corresponding defining equations. The arguments for the nonvanishing coordinates  $p_{I_a}$ and $p_{I_b}$ for general $1\leq a<b\leq {n-1}$ are similar.  It follows that   $X_0(z)$ satisfies all the defining equations of $F\ell_n$, and hence belongs to $F\ell_n$.
\end{proof}
\begin{remark}
 The vertices in the above statement are precisely the regular vertices of $\Delta$, namely the vertices that lie in exactly ${n(n-1)\over 2}$ facets of $\Delta$. There are exactly $n!=|S_n|$ regular vertices \cite[section 5.1]{Kiri}. The fibers of Gelfand-Cetlin system $F\ell_n\to \Delta$ were well studied in \cite{CKO, BMZ}, which are diffemorphic to the product of a real torus and a precisely described smooth manifold.  The regular vertices are exactly those points in $\Delta$ whose fibers are points.
\end{remark}

For any $v\in W^P\subset W=S_n$, the opposite Schubert variety $X^v$ of $X=F\ell_{n_1, \cdots, n_k; n}$
is given by the intersection of finitely many $W$-translated Schubert divisors $\{D_{p_{I_j}}\}_j$ with $X$.
For any $u\in W$, by $\pi_{u(I_j)}$ we mean the positive path in the ladder diagram of $X$ whose horizontal steps are precisely elements of the set $u(I_j)$. We consider the following subset of $\Delta$:
$$\Delta(u, v):=\bigcap_j\bigcup_{e\subset \pi_{u(I_j)}\atop e \mbox{ \tiny is effective}} F_e. $$
Notice that $X_v=w_0X^{\pi(w_0v)}$ where $w_0$ denotes the longest element of $W$, and $\pi(w_0v)$ denotes the unique element in $W^P$ such that $\pi(w_0v)^{-1} w_0v$ belongs to the subgroup $W_P$ generated by $\{s_j\mid j\in \{1,\cdots, n-1\}\setminus \{n_1,\cdots, n_k\}\}$.
\begin{thm}\label{thmTransInterviatoric}
    Let $v_1, \cdots, v_{m+1}, w\in W^P$ with $m\geq 1$ and  $\ell(w)=\sum_{i=1}^{m+1}\ell(v_i)$. Suppose that there exist
    $u_1,\cdots, u_{m+1}\in W$ such that  $\mathcal{S}:=\bigcap_{i=1}^{m+1} \Delta(u_i, v_i) \bigcap \Delta(w_0, \pi(w_0w))$ is a subset of $V^X$. Then $u_1 X^{v_1},\cdots, u_{m+1}X^{v_{m+1}}$,  $X_w$ intersect transversally, and $N_{v_1,\cdots, v_{m+1}}^w=\sharp \mathcal{S}$.
\end{thm}
\begin{proof} Denote by $u_{m+2}:=w_0$ and $v_{m+2}=\pi(w_0w)$. For $1\leq i\leq m+2$, we let  $$u_i\mathcal{X}^{v_i}:= \bigcap\nolimits_{I_{v_i}}\big(\{p_{u_i(I_{v_i})}=0\}\bigcap \mathcal{X}\big)$$
be the closed subvariety of $\mathcal{X}$ in which $\bigcap_{v_i}\{p_{I_{v_i}}=0\}$ defines the opposite Schubert variety $X^{v_i}$
of $X$.
Clearly, $\mathcal{Y}:=\bigcap_{i=1}^{m+2}u_i\mathcal{X}^{v_i}$  is a    closed subfamily of $\mathcal{X}$;
  $\mathcal{Y}|_{\mathbb{C}^*}$ is flat over $\mathbb{C}^*$, isomorphic to the trivial family $Y\times \mathbb{C}^*$ with $Y:=\bigcap_{i=1}^{m+2} u_iX^{v_i} $.  Hence, it admits a flat  extension $\overline{\mathcal{Y}|_{\mathbb{C}^*}}$,  sitting inside $\mathcal{Y}$,  by Proposition \ref{propHart}.
   It follows from the flatness and Theorem \ref{thmToricDeg} that
   $$\dim Y=\dim \overline{\mathcal{Y}|_{\mathbb{C}^*}}|_{t=1}=\dim \overline{\mathcal{Y}|_{\mathbb{C}^*}}|_{t=0}\leq \dim \mathcal{Y}|_{t=0}.$$
Notice that $\mathcal{Y}|_{t=0}$ is set-theoretically the toric subvariety of $X_0$ corresponding to $\mathcal{S}$.
In particular   if $\mathcal{S}$ is empty, then $\mathcal{Y}|_{t=0}=\emptyset$. Hence, we have $Y=\emptyset$ and consequently  $N_{v_1,\cdots, v_{m+1}}^w=0$ due to the geometrical meaning of the structure constants.

By Corollary \ref{coreqnvertices},  $X_0(z)$ is a coordinate point of the product of projective spaces, for any $z\in V$. By Theorem \ref{thmToricDeg}, the subspace of $X_0$ corresponding to $\mathcal{S}$ is given by  the set-theoretical  intersection of coordinate planes $\{p_I=0\}$ with $X_0$.
It follows that this subspace is given by intersections of the form $\{p_{I_j}=0\}_{j=1}^\ell$ and $\{p_{J_i}p_{K_i}=0\}_{i=1}^r$ with $J_i\neq K_i$ for all $i$. Hence, each point in the subvariety of $X_0$ corresponding to $\mathcal{S}$ is reduced in the scheme-theoretical intersection of these coordinate planes.
 Therefore, the cardinality of $Y$ with multiplicities counted is less than or equal to $\sharp \mathcal{S}$ due to the flatness of the family $\overline{\mathcal{Y}|_{\mathbb{C}^*}}$. Moreover, for any $z\in \mathcal{S}\subset V^X$, we have $X_0(z)\in X$ by definition. Hence, $X_0(z)$ is also in the intersection of the corresponding coordinate planes $\{p_I=0\}$ with $X$, so that $\sharp Y\geq \sharp \mathcal{S}$.
 It follows that $Y=\bigcap_{i=1}^{m+2} u_iX^{v_i} $  is a transversal intersection and $N_{v_1,\cdots, v_{m+1}}^w=\sharp Y=\sharp \mathcal{S}$. Moreover, $Y$ coincides with the subspace of $X_0$ corresponding to $\mathcal{S}$.
\end{proof}

\begin{remark}
  We expect that  $\mathcal{S}=\bigcap_{i=1}^{m+1} \Delta(u_i, v_i) \bigcap \Delta(w_0, \pi(w_0w))$ is always a subset of $V^X$, if $\mathcal{S}$ is a subset of $V$.
\end{remark}
\subsection{Towards a Littewood-Richardson rule} \label{covering} In this subsection, we will construct a modified partition of the following set
  $$A:=\{(u, v, w)\in S_n\times S_n\times S_n\mid \ell(w)=\ell(u)+\ell(v)\},$$
so that Conjecture \ref{mainconj00} makes sense.

Firstly, we divide the above set into the set
  $$ A_1:=\{(u, v, w)\in S_n\times S_n\times S_n\mid \ell(w)=\ell(u)+\ell(v), u\leq w, v\leq w\}$$
and its complement $A_0$. Here $u\leq w$ is with respect to the Bruhat order. It is a well-known fact that $N_{u, v}^w=0$ whenever $(u, v, w)\in A_0$.

Secondly, for every $(u, v, w)\in A_1$, we start with $B_{u, v, w}^{(1)}:=\{(u, v, w), (v, u, w)\}$, $B_{u, v, w}^{(2)}:=\{(u, v, w), ({w_0uw_0, w_0vw_0}, {w_0ww_0})\}$ and $B_{u, v, w}^{(3)}:=\{(u, v, w), ({u, w_0w}, {w_0v})\}$.
We notice the identities
 \begin{equation}\label{eqnid}
    N_{u, v}^w=N_{v, u}^w=N_{u, w_0w}^{w_0v}\mbox{  and  } N_{u, v}^w= N_{w_0uw_0, w_0vw_0}^{w_0ww_0}.
 \end{equation}
There is an automorphism of the Dynkin diagram of $SL(n, \mathbb{C})$ given by $\alpha_i\mapsto -w_0(\alpha_i)=\alpha_{n-i}$. It induces an automorphism of $X=F\ell_n=SL(n, \mathbb{C})/B$, which sends $X^{v}$ to $X^{w_0vw_0}$. Then induced automorphism of the cohomology $H^*(F\ell_n, \mathbb{Z})$ results in the last identity in \eqref{eqnid}.

There is another type of identities among  (equivariant) genus zero, three-point Gromov-Witten invariants for complete flag variety $G/B$ of general Lie types in \cite{LeLi, HuLi}. The special case of degree zero Gromov-Witten invariants recovers the recurrence formula in \cite{Knut}. For classical cohomology of $F\ell_n$, it says the following.
\begin{prop}\label{propreduction} {For any } $u, v, w\in S_n$ { and any } $1\leq i\leq n-1$,  we have
    {\upshape  $$
     N_{u, v}^w=\begin{cases}
      N_{us_i, v}^{ws_i},& \mbox{ if }\ell(us_i)>\ell(u), \ell(vs_i)>\ell(v), \ell(ws_i)>\ell(w),\\
           0,& \mbox{ if }\ell(us_i)>\ell(u), \ell(vs_i)>\ell(v), \ell(ws_i)<\ell(w).\\
   \end{cases}
  \hspace{0cm}$$}
  Consequently, we consider $$B_{u, v, w}^{(4)}:=\{(u, v, w)\}\cup \{(us_i, v, ws_i)\mid \ell(us_i)>\ell(u), \ell(vs_i)>\ell(v), \ell(ws_i)>\ell(w) \mbox{ for some } i\}.$$
   \end{prop}
In this way, we obtain a covering of $A$, given by
  $$\mathcal{U}_0:=\{A_0\}\bigcup \{B_{u, v, w}^{(j)}\mid (u, v, w)\in A_1, 1\leq j\leq 4\}$$
It defines a consistent relation  $R_0\subset A\times A$, and hence gives rise to an equivalence relation $R_1\subset A\times A$ by letting $R_1$ be the transitive closure of $R_0$. The quotient space $\mathcal{U}_1:=A/R_1$ is a partition of $A$ with $A_0$ being an equivalence class in $\mathcal{U}_1$. Now we set
 $$C:=\{[(\hat u, \hat v, \hat w)]\in \mathcal{U}_1\setminus \{A_0\}\left| {\mbox{there is }(u, v, w)\in [(\hat u, \hat v, \hat w)] \mbox{ such that all the inequalities} \atop \ell(us_i)>\ell(u), \ell(vs_i)>\ell(v), \ell(ws_i)<\ell(w) \mbox{hold for some }i}\right.\}.$$
It follows that
   $$\mathcal{U}:=\Big(\mathcal{U}_1\setminus (C\cup \{A_0\})\bigcup \{A_0\cup \cup_{\alpha\in C} \alpha\}$$
is still a partition of $A$. Moveover, $N_{u_1, v_1}^{w_1}=N_{u_2, v_2}^{w_2}$ whenever $[(u_1, v_1, w_1)]=[(u_2, v_2, w_2)]$; $N_{u_1, v_1}^{w_1}=0$ whenever $(u_1, v_1, w_1)\in A_0\cup \cup_{\alpha\in C} \alpha$.

Finally we construct $\hat{\mathcal{U}}$ as follows. We say that $u\in S_n$ satisfies propery $(\star)$ at level $m$, if
$u$ can be written as $u=u_1u_2\cdots u_m$ with $\ell(u_i)\geq 1$ and
 the sets  $\Xi_i:=\{s_k\mid s_k \mbox{ occurs in a reduced expression of } u_i\}$ satisfy  $\Xi_i\cap \Xi_j=\emptyset$ and $ss'=s's$ for any $s\in \Xi_i, s'\in \Xi_j$, whenever $i\neq j$. We have the following equality.
 \begin{prop}\label{propred}
   Let $u\in S_n$. If $u=u_1u_2$ satisfies property $(\star)$ at level $2$, then $\sigma^{u_1}\cup \sigma^{u_2}=\sigma^u$.
\end{prop}
\begin{proof}
 For $w\in S_n$, $N_{u_1, u_2}^w\neq 0$ only if $u_i\leq w$ with respect to the Bruhat order for $i=1, 2$. That is, $u_i$ can be obtained from a subexpression of a reduced expression of $w$.
 Then it follows from the definition of property $(\star)$ that $w=u_1u_2$ and  $\ell(w)=\ell(u_1)+\ell(u_2)$. By Proposition \ref{propreduction}, we conclude $N_{u_1, u_2}^{u_1u_2}=N_{u_1, \rm id}^{u_1}=1$. Therefore the statement follows.
\end{proof}
\noindent We remark that many Schubert classes cannot be represented as a positive multiple of Schubert classes in $H^{>0}(F\ell_n)$. For instance, $\sigma^{s_1s_2s_3}$ does not satisfy property $(\star)$ at level $2$ nor at level $3$.

Now for any $[(\hat u, \hat v, \hat w)]\in \mathcal{U}\setminus \{A_0\cup \cup_{\alpha\in C} \alpha\}$, we check every $(u, v, w)\in [(\hat u, \hat v, \hat w)]$. If $u$ satisfies property $(\star)$ at some level larger than $1$, then we take the maximal level $m$  at which $u=u_1\cdots u_m$ satisfies property $(\star)$. So does for $v=v_1\cdots v_{m'}$. Then we add $(u_1, \cdots, u_m, v_1, \cdots, v_{m'}, w)$ to the set $[(\hat u, \hat v, \hat w)]$.
We denote by   $\hat{\mathcal{U}}$ the union of $\{A_0\cup \cup_{\alpha\in C} \alpha\}$ and all the modified classes $[(\hat u, \hat v, \hat w)]$, and call it  a modified partition of $A$.
In a summary, elements $[(u, v, w)]$ in $\hat{\mathcal{U}}$ satisfy either of the following.
\begin{enumerate}
  \item[(a)]  $[(u, v, w)]=A_0\cup \cup_{\alpha\in C} \alpha$; every element in this class is a triple $(u', v', w')$ and  $N_{u', v'}^{w'}=0$.
  \item[(b)] Elements in $[(u, v, w)]$ are $(m+m'+1)$-tuples  $(u_1, \cdots, u_m, v_1, \cdots, v_{m'}, w')$, for most of which $m=m'=1$;
  $N_{u_1, \cdots, u_m, v_1,\cdots, v_{m'}}^{w'}=N_{u, v}^w$.
\end{enumerate}
Now we rewrite both $(u', v', w')$ and $(u_1, \cdots, u_m, v_1, \cdots, v_{m'})$ uniformly in the form $(u_1', \cdots, u_{m+1}')$, and restatement Conjecture \ref{mainconj00} as follows.
\begin{conjecture}\label{mainconj33}
   For any  $[({u, v},{w})]\in \hat{\mathcal{U}}$,  there exists $(u_1',\cdots, u_{m+1}',{w'})\in [({u, v},{w})]$ such that
     $\tilde u_1\cdot X^{u'_1}\cap\cdots\cap\tilde u_{m+1} \cdot  X^{u'_{m+1}}\cap X_{w'}$ is a  transversal intersection   for some
    $\tilde u_1,\cdots, \tilde u_{m+1}\in S_n$.
\end{conjecture}
\noindent We notice that the above conjecture holds for  case (a). In this case, $[(u, v, w)]$ contains an element $(u', v', w')\in A_0$, for which $u'\not\leq w'$ or $v'\not\leq w'$ holds; consequently either of the intersections  $X^{u'}\cap X_{w'}$ and $X^{v'}\cap X_{w'}$ is an emptyset already.

\subsection{Special Schubert structure constants for $Gr(k, n)$}
The partitions  $1^r:=(1, 1, \cdots, 1, 0, \cdots, 0)$ and  $r:=(r, 0,\cdots, 0)$ for $Gr(m, n)$ are special. Geometrically,
the Schubert class  $\sigma^{r}=P.D.[X^{r}]$ (resp. $\sigma^{1^r}=P.D.[X^{1^r}]$)   is the $r$-th Chern class of the tautological quotient bundle (resp. dual subbundle) over the $Gr(m, n)$.
Recall that a partition $\mu$ corresponds to a positive path $\pi_\mu$ and hence a Grassmannian permuation  $w_\mu \in W^P$.
\begin{prop}\label{propsp1Grk}
   For any partitions $\mu, \eta\in \mathcal{P}_{m, n}$ satisfying $\mu\leq \eta$ and  $|\eta|=|\mu|+1$,
   the intersection    $w_\mu X^1\bigcap X^\mu \bigcap X_\eta$ is transversal and  consists of a point.
\end{prop}
 \begin{proof}
  The argument  can be read off immediately from Figure \ref{figureGrassInt}.
    \begin{figure}[h]
  \caption{Triple intersections for $Gr(5, 8)$}\label{figureGrassInt}
  \bigskip
     \includegraphics[scale=0.85]{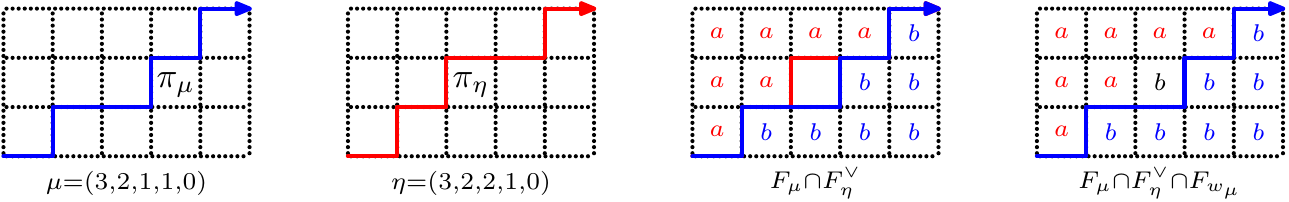}
  \end{figure}
  Indeed, the face $\Delta({\rm id}, w_\mu)$ (resp. $\Delta(w_0, \pi(w_0w_{\eta})$) of $\Delta$ corresponds to the toric subvariety to which $X^{w_\mu}$ (resp. $w_0X^{\pi(w_0w_{\eta})}=X_\eta$) degenerates. Thus it is given by the face $F_\mu$ (resp. $F^\vee_\eta$) by Proposition \ref{propTorDegSchuGrass}.   The intersection   $F_{\eta}^\vee\cap F_\mu$ is a one-dimensional face  with a parameter in the un-valued box bounded by the positive paths $\pi_\mu$ and $\pi_\eta$ (see the third diagram in Figure \ref{figureGrassInt} for an illustration). All the boxes above $\pi_\eta$ (resp. below $\pi_\mu$) take value $a=\lambda^{(n)}_m$ (resp. $b=\lambda^{(n)}_n$). Notice that $w_\mu X^{s_m}=w_\mu X^1$ and $\Delta(w_\mu, s_m)=\bigcup_{e\subset \pi_\mu}  {F_e}$. Thus   any effective edge $e$ on $\pi_\mu$ is the  common edge of a  $b$-valued box with either an $a$-valued box or   the unvalued box. In the former case,  the intersection $F_{\eta}^\vee\cap F_\mu\cap F_e$ is empty;  in the latter case, the intersection
     results in a same vertex of the latter diagram  by evaluating $b$ on the unvalued box. Hence, the set    $\Delta(w_\mu, s_m)\bigcap \Delta({\rm id}, w_\mu)\bigcap \Delta(w_0, \pi(w_0w_\eta))$ consists of a point in $V$, and therefore
       the statement follows by Theorem \ref{thmTransInterviatoric}.
      \end{proof}

 \begin{cor}[Pieri-Chevalley formula for $Gr(m, n)$]
   For any partition $\mu \in \mathcal{P}_{m, n}$, $$\sigma^{1}\cup \sigma^\mu=\sum_{|\eta|=|\mu|+1,\,\, \eta\geq \mu } \sigma^\eta.$$
 \end{cor}
 \begin{proof}
 Notice that $\sigma^\eta$ occurs in $\sigma^1\cup \sigma^\mu$ only if $\eta\geq \mu$ and $|\eta|=|\mu|+1$, for which    the Richardson variety $X^\mu\cap X_\eta$ is nonempty. We are done by Proposition \ref{propsp1Grk}.
 \end{proof}

\begin{prop}\label{propspeciala0Gr}
Let $r, q\in \mathcal{P}_{m, n}$ be special partitions  with    $r+q\leq n-m$. There exists $u\in S_n$ such that the intersection   $uX^r\bigcap X^q\bigcap X_{r+q}$
    is transversal and consists of  a point. In particular $N_{r, q}^{r+q}=1$.
\end{prop}
\begin{proof}
  Take $u\in S_n$ such that $u(j)=\begin{cases}
     j,&\mbox{if } 1\leq j\leq m-1;\\
     j+q,&\mbox{if } m\leq j\leq m+r-1.
  \end{cases}$
Denote by $w_r$ the Grassmannian permutation of the special partition $r=(r, 0,\cdots, 0)$ in $\mathcal{P}_{m, n}$.
  Notices that $\Delta({\rm id}, w_q)\bigcap\Delta(w_0, \pi(w_0w_{r+q}))=F_q\cap F^\vee_{r+q}$ is the face of $\Delta$ that takes value $a$ on every box above the positive path $\pi_{1,2,\cdots, m-1, m+r+q}$ and takes value $b$ on every box below the positive path $\pi_{1,2,\cdots, m-1, m+q}$.
Clearly,   the intersection of
  $F_q\cap F^\vee_{r+q}$ with 
   $$\bigcup_{e\subset \pi_{1,\cdots, m-1, m+q}}F_e=\{\lambda_{m}^{(m+q)}=\lambda_m^{(m+q+1)}\}\bigcup\bigcup_{i=1}^q\{\lambda_{m-1}^{(m-1+i)}=\lambda_{m}^{(m+i)}\},$$
gives the face $F_{q+1}\cap F^\vee_{q+r}$, the union in which consists of facets given by $\pi_{1,\cdots, m-1,j+q}$.
By induction,  the intersection of
  $F_q\cap F^\vee_{r+q}$ with the intersection of union of facets given by $\{\pi_{1,\cdots, m-1, j+q}\mid m\leq j\leq m+r-1\}$
  gives the vertex $F_{r+q}\cap F^\vee_{r+q}$. It is easy to see that this point belongs to the union of facets given by any one of the remaining positive path $\pi_{u(I)}$ satisfying $\pi_I\not \neq \pi_{1,\cdots, m-1,m+q}$.
  Therefore,  $\Delta(u, w_r)\bigcap\Delta({\rm id}, w_q)\bigcap\Delta(w_0, \pi(w_0w_{r+q}))=F_{r+q}\cap F^\vee_{r+q}$ consists of a vertex of $\Delta$. Hence, the statement follows from Lemma \ref{lemverticesGr} and Theorem \ref{thmTransInterviatoric}.
\end{proof}
\begin{remark}\label{rmkGr1n}
  The  case   $m=1$ tells that   Conjecture \ref{mainconj00} holds for $Gr(1, n)=\mathbb{P}^{n-1}$.
\end{remark}

\subsection{Transversal intersections in $Gr(2, n)$}
Recall that $C$ denotes the cyclic permutation $(2, 3,\cdots, n, 1)$. We will see in Theorem \ref{thmGr2n} that Conjecture \ref{mainconj00} holds for $X=Gr(2, n)$.
\begin{prop}\label{propPierGr2n}
  Let $\mu, \eta\in \mathcal{P}_{2, n}$. If  $\eta=\mu+(1,1)$, then
   $X^{(1,1)}\bigcap C\cdot  X^\mu\bigcap X_\eta$ is a transversal intersection  and consists of a point. If $\eta-\mu=(2,0)$ or $(0, 2)$, then
     $C^k \cdot X^{(1,1)}\bigcap X^\mu\bigcap X_\eta=\emptyset$ for some $k\in \mathbb{Z}$.
\end{prop}
\begin{proof}
   If $(\eta_1, \eta_2)=(\mu_1, \mu_2)+(1, 1)$, then we conclude that    $ X^{(1,1)}\bigcap C\cdot X^\mu\bigcap X_\eta=$ $\{p_I=0\mid \pi_I\neq \pi_\eta\}$, which is a reduced coordinate point of $\mathbb{P}^{{n\choose 2}-1}$ so that the first statement follows.
   Indeed, we notice that $X_\eta=X\cap\{p_I=0\mid \pi_I\not\leq \pi_\eta\}$,  $C X^\mu=X\cap\{p_{C(I)}=0\mid \pi_I\not\geq \pi_\mu\}$, and $X^{(1, 1)}=X\cap \{p_{1j}=0\mid 2\leq j\leq n\}$.
   Thus $p_I=0$ with $I=(i_1, i_2)$ does not occur in the triple intersection if and only if $\pi_I\leq \pi_\eta, i_1>1$ and $\pi_{(i_1-1, i_2-1)}\geq \pi_\mu$ all hold; that is, $(i_1, i_2)\leq (\eta_1, \eta_2), i_1>1, (i_1-1, i_2-1)\geq (\mu_1, \mu_2)=(\eta_1,\eta_2)- (1,1)$.  Hence, $I=\eta$ and we are done.

   Now we assume $(\eta_1, \eta_2)=(\mu_1, \mu_2)+(2, 0)$ and let $v=C^{\mu_2}$. Then $\mu_1+2\leq n-2$. Notice that  $v X^{(1, 1)}=X\cap \{p_{1+\mu_2,j}=0\mid 1\leq j\leq n, j\neq  1+\mu_2\}$, and that $F_\mu\cap F_\eta^\vee$ takes value $a$ (resp. $b$) on each box above (resp. below)
   the positive path $\pi_\eta$ (resp. $\pi_\mu$), and the only   unvalued boxes are labeled by  $\lambda_2^{(\mu_1+3)}$ and  $\lambda_2^{(\mu_1+4)}$.  Thus  for any effective edge $e$ on the positive path $\pi_\mu=\pi_{(1+\mu_2, 2+\mu_1)}$, the intersection of the facet $F_e$ with  $F_\mu\cap F_\eta^\vee$ is empty unless $e$ is the common edge of the boxes labelled by $\lambda_2^{(\mu_1+2)}$ and $\lambda_2^{(\mu_1+3)}$, which forces the $\lambda_2^{(\mu_1+3)}$-box to take value $b$.
   The further intersection with the facets determined by  $\pi_{(1+\mu_2, 3+\mu_1)}$ will make the $\lambda_2^{(\mu_1+4)}$-box to take value $b$. This will result in an empty set after one more intersection with the facets  determined by  $\pi_{(1+\mu_2, 4+\mu_1)}$.
   Therefore we have  $\Delta(v, w_{(1,1)})\cap \Delta(\mbox{id}, w_{\mu})\cap \Delta(w_0, \pi(w_0w_{\eta}))=\emptyset\subset \Delta$.
   If $(\eta_1, \eta_2)=(\mu_1, \mu_2)+(0, 2)$, then  we can also conclude the empty intersection    by similar arguments with respect to $v=C^{2+\mu_1}$. 
    It follows from Theorem \ref{thmTransInterviatoric} that  $vX^{(1,1)}\bigcap X^\mu\bigcap X_\eta=\emptyset$ in these two cases.
\end{proof}

\begin{cor}[Pieri rule for $Gr(2, n)$]\label{corPierGr2n}
Let $\mu=(\mu_1,\mu_2)\in \mathcal{P}_{2, n}$. We have  $\sigma^{(1,1)}\cup \sigma^{(\mu_1,\mu_2)}=\sigma^{(\mu_1+1,\mu_2+1)}$ if $\mu_1<n-2$, or $0$ otherwise.
\end{cor}
\begin{proof}
  $\sigma^{\eta}$ occurs in   $\sigma^{(1,1)}\cup \sigma^{(\mu_1,\mu_2)}$ only if $n-2\geq \eta_2\geq \eta_1\geq 0$,  $\eta\geq \mu$, $\eta\geq (1, 1)$ and $|\eta|=|\mu|+2$. Therefore $\eta-\mu$ can only be given by
   $(2, 0), (1,1)$ or $(0, 2)$. Hence, the statement follows from Proposition \ref{propPierGr2n}.
   \end{proof}

\begin{thm}\label{thmGr2n}
  For any $\lambda, \mu, \eta\in \mathcal{P}_{2, n}$, there exist $\lambda', \mu', \eta'\in \mathcal{P}_{m, n}$ for some $m$, such that
  $N_{\lambda, \mu}^\eta=N_{\lambda', \mu'}^{\eta'}$  and   $uX^{\lambda'}\bigcap X^{\mu'}\bigcap X_{\eta'}$ is a transversal intersection in $Gr(m, n)$ for some $u\in S_n$.
\end{thm}

\begin{proof}
   By Corollary \ref{corPierGr2n}, we have $N_{\lambda, \mu}^\eta=N_{(a, 0), (b,0)}^{(c, d)}$ where $a=\lambda_1-\lambda_2, b=\mu_1-\mu_2$, $c=\eta_1-\lambda_2-\mu_2$ and $d=\eta_2-\lambda_2-\mu_2$.
   Clearly, $N_{(a, 0), (b,0)}^{(c, d)}$ is nonzero only if $n-2\geq c\geq d\geq 0$, $c+d=a+b$ and $c\geq \max\{a, b\}$, which implies $d\leq \min\{a, b\}$.
   Notice that $u=s_{a+1}s_a\cdots  s_2$, $v=s_{b+1}s_b\cdots s_2$ and $w=s_{d}\cdots s_1s_{c+1}s_c\cdots s_2$ are the Grassmannian permutations associated to the partitions $(a, 0), (b, 0)$ and $(c, d)$ respectively.
   By Proposition \ref{propreduction}, we have
   $$
       N_{u, v}^w =N_{us_1,v}^{ws_1}=N_{us_1, vs_2}^{ws_1s_2}=N_{u, vs_2}^{ws_1s_2s_1}=N_{us_2, vs_2}^{ws_1s_2s_1s_2}
                  =N_{s_{a+1}\cdots s_3, s_{b+1}\cdots s_3}^{s_d\cdots s_2 s_{c+1}\cdots s_3}.
$$
The elements ${s_{a+1}\cdots s_3, s_{b+1}\cdots s_3}, {s_d\cdots s_2 s_{c+1}\cdots s_3}$ are Grassmannian permutations associated to the partitions $(a-1, 0, 0), (b-1, 0,0)$ and $(c-1, d-1, 0)$ for $Gr(3, n)$.
By induction and using Proposition \ref{propreduction} repeatedly, we conclude that $N_{\lambda, \mu}^\eta=N_{u, v}^w$ is equal to the Schubert structure constant
   $N_{\lambda', \mu'}^{\eta'}$ for $Gr(d+2, n)$ with $\lambda'=(a-d, 0,\cdots, 0), \mu'=(b-d, 0, \cdots, 0)$ and $\eta'=(c-d, 0, \cdots, 0)$. Hence, the statement follows from Proposition \ref{propspeciala0Gr}.
\end{proof}


\subsection{Discussions for complete flag varieties}

A face of the Gelfand-Celtin polytope $\Delta$ for complete flag variety $F\ell_n$   is called a (\textit{dual}) \textit{Kogan face}, if all the equalities are of the form $\lambda^{(i)}_j=\lambda^{(i-1)}_j$ (resp. $\lambda^{(i)}_j=\lambda^{(i+1)}_{j+1}$). In other words, a (dual) Kogan face $F$ (resp. $F^*$) is indexed by a set of horizontal (resp. vertical) effective edges of the ladder diagram $\Lambda$ (which were defined in terms of the dual graph of $\Lambda$ in \cite{KST}).
For each edge on the bottom (resp. right) of a box on the $i$th row (resp. column) where $1\leq i\leq n-1$, we assign the simple reflection $s_i$. Then every Kogan face $F$   defines a permutation $w(F)$  by going from bottom to top   and then by going from left to right among those edges of $\Lambda$ that define  $F$. Similarly,   every dual Kogan face $F^*$   defines a permutation $w(F^*)$  by going from left to right and then by going from bottom to top   among those edges  defining $F^*$. A (dual) Kogan face is called reduced, if the resulting expression is a reduced decomposition of   the associated permutation.
For example, the first and third faces in Figure \ref{figureKoganfaces} are reduced, while the second Kogan face is not.
Recall that we denote by  $X_F$   the toric subvariety of $X_{\Delta}$ defined by the face $F$ of $\Delta$, and notice that the number of edges defining a (dual) Kogan face equals the codimension of the face in the Gelfand-Cetlin polytope.  Thanks to this identification, we have the following.
\begin{figure}[h]
  \caption{(dual) Kogan faces and the associated permutations}\label{figureKoganfaces}
  \bigskip
   \includegraphics[scale=1]{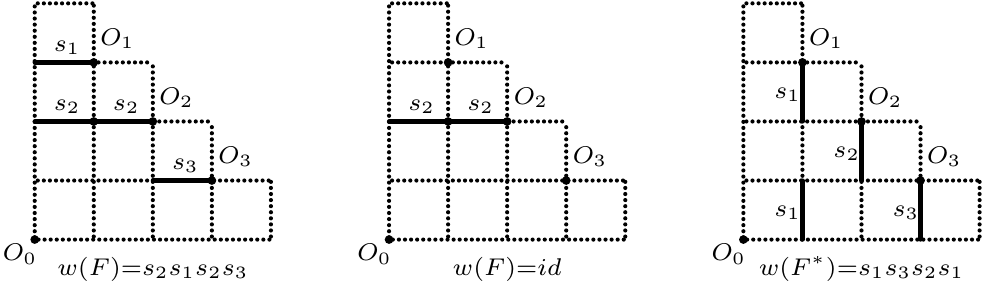}
   \end{figure}

\begin{prop}[Theorem 8 and Remark 10 of \cite{Kogan}]\label{propKMtoricSchub}
  The flat family $\mathcal{X}$ contains two kinds of flat subfamilies, giving the flat degeneration of (opposite) Schubert varieties $X_u$, $X^v$ respectively to the reduced union of toric subvarieties of reduced (dual) Kogan faces {\upshape $$\sum_{w(F)=w_0u}X_F,\qquad \sum_{w(F^*)=v}X_{F*}.$$ }
\end{prop}

\begin{example}\upshape Let us consider the Schubert structure constant $N_{\lambda, \mu}^\eta$ of $Gr(3, 6)$ where $\lambda=\mu=(2,1,0)$ and $\eta=(3,2,1)$.
As $X_\eta=w_0X^\mu$, this is the most complicated case in the sense that it counts the cardinality of the triple self-intersection  $g_1X^\mu\cap g_2X^\mu\cap w_0X^\mu$ up to a generic translation.
 For any $u, v\in S_6$, it follows from direct calculations that  $uX^\mu\cap vX^\mu\cap X_\eta$ is of positive dimension, and hence it cannot be a transversal intersection.
Nevertheless, we have the following equalities.  The first equality
 is simply the identification of permutations with Grassmanniannian permutations and the natural treatment of Structure constants for $Gr(3, 6)$ as that for $F\ell_6$. The seventh equality applies  identity (2) in section 4.2, where we notice $w_0s_iw_0=s_{6-i}$ in this example. The eighth equality uses the fact $\sigma^{s_2}\cup \sigma^{s_4}=\sigma^{s_2s_4}$. The remaining equalities follow from  Proposition \ref{propreduction}.
\begin{align*}
   N_{\lambda, \mu}^\eta=N_{s_2s_4s_3, s_2s_4s_3}^{s_3s_5s_4s_1s_2s_3}&=N_{s_2s_4s_3, s_2s_4s_3s_1}^{s_3s_5s_4s_1s_2s_3s_1}\\
                          &=N_{s_2s_4s_3, s_2s_4s_3s_1s_2}^{s_3s_5s_4s_1s_2s_3s_1s_2}\\
                          &=N_{s_2s_4, s_2s_4s_3s_1s_2}^{s_3s_5s_4s_2s_3s_1s_2}\\
                           &=N_{s_2s_4, s_2s_4s_3s_1s_2s_3}^{s_3s_5s_4s_2s_3s_1s_2s_3}\\
                            &=N_{s_2s_4, s_2s_4s_3s_1s_2s_3s_1}^{s_3s_5s_4s_2s_3s_1s_2s_3s_1}\\
                            &=N_{s_4s_2, s_4s_2s_3s_5s_4s_3s_5}^{s_3s_1s_2s_4s_3s_5s_4s_3s_5}\\
                            &=N_{s_4, s_2, s_4s_2s_3s_5s_4s_3s_5}^{s_3s_1s_2s_4s_3s_5s_4s_3s_5}.
\end{align*}
\noindent  Denote $v=s_4s_2s_3s_5s_4s_3s_5$ and $w=s_3s_1s_2s_4s_3s_5s_4s_3s_5$.
There is a unique subexpression $s_2s_3s_4s_5s_3s_4s_3$ of $w_0=s_1s_2s_3s_4s_5s_1s_2s_3s_4s_1s_2s_3s_1s_2s_1$ of length $7$ whose product equals $v$, given by the $\{2,3,4,5,8,9,12\}$th positions. This gives the unique reduced dual Kogan face $F^*$ together with the toric subvariety $X_{F^*}$, to which    $X^{v}$ degenerates.
There is a unique subexpression of $w_0=s_5s_4s_3s_2s_1s_5s_4s_3s_2s_5s_4s_3s_5s_4s_5$ of length $(15-9)$ whose product equals $w_0w$, given by the $\{2,3,4,5,8,9\}$th positions. This gives the unique reduced Kogan face $F$ together with the toric subvariety $X_{F}$, to which $X_{w}$ degenerates. The faces $F^*$, $F$ are precisely given by $\Delta(\mbox{id}, v)$ and $\Delta(w_0, w_0w)$ respectively in Figure \ref{figuredimTwoSch}. Let $\tilde u$ (resp. $\tilde v$) denote the Grassmannian permutation  associated to   the positive path $\pi_{\{2, 3\}}$ (resp. $\pi_{\{1,4,5,6\}}$).
The intersection of $F^*\cap F$ with $\Delta(\tilde u, s_2)$ (i.e. with the union of faces corresponding to the $W$-translated Schubert divisor $\tilde uX^{s_2}$) is the union of two one-dimensional faces of $\Delta$.
The intersection $F^*\cap F\cap\Delta(\tilde u, s_2)\cap \Delta(\tilde v, s_4)$ consists of two regular vertices of $\Delta$ for $F\ell_6$. Therefore it follows from Theorem \ref{thmTransInterviatoric} that
$\tilde u X^{s_2}\cap \tilde v X^{s_4}\cap X^v\cap X_w$ is a transversal intersection and consists of two regular vertices of $\Delta$, and consequently $N_{s_4, s_2, v}^w=2$.
   \begin{figure}[h]
  \caption{Intersections  corresponding to $\tilde u X^{s_2}\cap \tilde v X^{s_4} \cap X^v\cap X_w$ in $F\ell_6$}\label{figuredimTwoSch}
  \bigskip
     \includegraphics[scale=1]{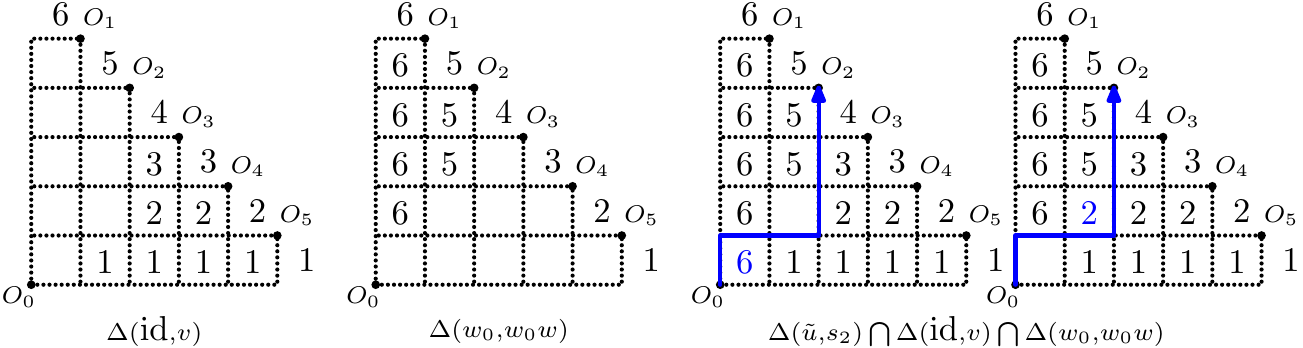}
     \begin{center}
              \includegraphics[scale=1]{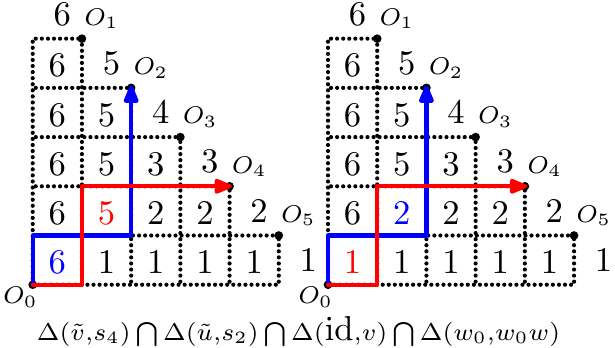} \end{center}

  \end{figure}
\end{example}

\begin{example}{\upshape
   Let $X=F\ell_4$. It suffices to consider the Schubert structure constants $N_{u, v}^w$ where $u, v, w\in S_4$ satisfying $\ell(w)=\ell(u)+\ell(v)$ and $u\leq w, v\leq w$ with respect to the Bruhat order. The case for $u=\mbox{id}$ or $v=\mbox{id}$ is trivial.  Noting that $X^u=w_0X_{w_0u}$ is of dimension $(6-\ell(u))$, we have $N_{u, v}^w=N_{w_0w, v}^{w_0u}$. Thus   we can also assume $\ell(w)\leq 4$ and use the      identity (2) (i.e. $N_{u, v}^{w}=N_{w_0uw_0, w_0vw_0}^{w_0ww_0}$) in section 4.2, where $w_0uw_0$ is obtained from an reduced expression of $u$ by interchanging $s_1$ and $s_3$.
    Then we do a further reductions by using the  Proposition \ref{propreduction}.   It turns out that all the remaining $N_{u, v}^w$ are either equal to $0$ (for instance we have $N_{s_1, s_1s_2s_1}^{s_2s_3s_1s_2}=N_{s_1, s_2s_1}^{s_2s_1s_3}=0$) or reduced to one of the following cases.     }
\end{example}
\begin{enumerate}

   \item  $N_{s_1, s_2s_3s_2}^{s_2s_3s_1s_2}=N_{s_1, s_2s_3}^{s_2s_3s_1}=N_{s_1, s_2}^{s_2s_1}=N_{s_1, s_2s_3}^{s_2s_1s_3}=N_{s_1s_2, s_2s_3}^{s_2s_1s_3s_2}$;

     \noindent  $N_{s_1, s_2s_3s_2}^{s_1s_2s_3s_2}=N_{s_1, s_2s_3}^{s_1s_2s_3}=N_{s_1, s_2}^{s_1s_2}=N_{s_1, s_2s_3}^{s_1s_2s_3}=N_{s_1s_2, s_2s_3}^{s_1s_2s_3s_2}$;
  \begin{align*}
      N_{s_1, s_3s_1s_2}^{s_3s_1s_2s_1}&=N_{s_1, s_3s_1}^{s_3s_2s_1}=N_{s_1s_2, s_3s_1}^{s_3s_2s_1s_2}\\&=N_{s_1s_2, s_3s_1}^{s_3s_2s_1s_2}=N_{s_1s_2, s_3}^{s_3s_1s_2}=N_{s_1, s_3}^{s_1s_3}=N_{s_1, \scriptstyle\rm id}^{s_1}=N_{s_1, s_2}^{s_1s_2};
  \end{align*}
  $N_{s_2, s_2s_3s_1}^{s_2s_3s_1s_2}=N_{{\rm id}, s_2s_3s_1}^{s_2s_3s_1}=N_{{\rm id}, s_2}^{s_2}=N_{s_1, s_2}^{s_2s_1}$. 

  \item $\mbox{}$
 \vspace{-0.5cm}
  \begin{align*}
 N_{s_1, s_3s_1s_2}^{s_2s_3s_1s_2}&=N_{s_1, s_3s_1}^{s_2s_3s_1}=N_{s_1, s_1}^{s_2s_1}=N_{s_1,s_1s_2}^{s_2s_1s_2}=N_{s_1, s_1s_2s_3}^{s_2s_1s_2s_3}\\
             &=N_{s_1s_2, s_3s_1}^{s_2s_3s_1s_2}\qquad\quad\,\,\!=N_{s_1s_3,s_1s_2}^{s_2s_1s_2s_3},
  \end{align*}
  $$N_{s_1, s_3s_1s_2}^{s_2s_3s_1s_2}=N_{s_1s_3, s_3s_1s_2}^{s_2s_3s_1s_2s_3}=N_{s_1s_3, s_3s_1}^{s_2s_1s_3s_2},\quad N_{s_1, s_1s_2s_3}^{s_2s_1s_2s_3}=N_{s_1s_2, s_1s_2s_3}^{s_2s_1s_2s_3s_2}=N_{s_1s_2, s_1s_2}^{s_2s_1s_3s_2};$$
       \begin{align*}
       N_{s_1, s_1s_2s_1}^{s_3s_1s_2s_1}&=N_{s_1, s_2s_1}^{s_3s_2s_1}=N_{s_1s_2, s_2s_1}^{s_3s_2s_1s_2}=N_{s_1s_2s_3, s_2s_1}^{s_3s_2s_1s_2s_3}=N_{s_1s_2s_3, s_2}^{s_3s_1s_2s_3} \\
                &=N_{s_1s_3, s_2s_1}^{s_3s_2s_1s_3}=N_{s_1s_3s_2, s_2s_1}^{s_3s_2s_1s_3s_2}=N_{s_1s_3s_2, s_2}^{s_2s_3s_1s_2}.\\
    \end{align*}
   \vspace{-1cm}
     \item  $N_{s_1s_2, s_3s_2}^{s_2s_1s_3s_2}$;
   \begin{align*}
      N_{s_2, s_2s_3s_1}^{s_1s_2s_3s_1}&=N_{s_2, s_2s_3}^{s_1s_2s_3}=N_{s_2, s_2}^{s_1s_2}=N_{s_2s_3, s_2}^{s_1s_2s_3}=N_{s_2s_3, s_2s_1}^{s_1s_2s_3s_1}=N_{s_2s_3, s_2s_1s_2}^{s_1s_2s_3s_1s_2}=N_{s_2, s_2s_1s_2}^{s_2s_1s_3s_2}\\
          &=N_{s_2, s_2s_3s_2}^{s_1s_2s_3s_2}.
   \end{align*}

\end{enumerate}
In case (1), it is reduced to the trivial case $N_{{\rm id}, w}^w=1$. In case (2), it is reduced to the Schubert structure constants $N_{s_1, s_1}^{s_2s_1}$ and $N_{s_1, s_2s_1}^{s_3s_2s_1}$ for $Gr(1, 4)$, and hence is done by Remark \ref{rmkGr1n}. In case (3),
it is reduced to the Schubert structure constants $N_{s_2, s_2}^{s_1s_2}$ and $N_{s_1s_2, s_3s_2}^{s_2s_1s_3s_2}$ for $Gr(2, 4)$, and hence is done by Propositions \ref{propsp1Grk} and \ref{propPierGr2n}.
The above identities result  in a consistent relation on the set $A$. Without going further to the modified partition $\hat{\mathcal{U}}$, we   have already verified the transversality of the corresponding triple intersections. Therefore  Conjecture \ref{mainconj00} holds for $F\ell_4$.

\section{An anti-canonical divisor $-K_{X}$}
As an   application of  $W$-translated Schubert divisors, we specify an anti-canonical divisor $-K_X$ of $X=F\ell_{n_1, \cdots, n_k; n}=SL(n, \mathbb{C})/P$,
which  may have potential applications in the study of
 Strominger-Yau-Zaslow
 mirror symmetry for   $X$.  For instance for $X=F\ell_{1, n-1; n}$, a special Lagrangian fibration for the open Calabi-Yau manifold $X\setminus -K_X$ was constructed in \cite{CLL} with respect to such $-K_X$.

 The  anti-canonical divisor class 
 is given by    (see \cite[Lemma 3.5]{FuWo} for the formula of $c_1(T_{G/P})$ for general Lie types)
 \begin{align*}
    [-K_X]=c_1(T_X)=\sum_{i=1}^k(n_{i+1}-n_{i-1})\sigma^{s_{n_i}}.
\end{align*}

\noindent By Proposition \ref{propWSchubertdivisor}, all positive paths towards $O_i$ represent the same Schubert divisor class $\sigma^{s_{n_i}}$.
Now we introduce the following notion, in order to   specify a representative of $[-K_X]$ with a property compatible with the toric degeneration of $X$.
\begin{defn}
  Let  $e$ be an edge   on the roof of $\Lambda$. We define the \textbf{special path} associated with $e$ to be the (unique)  positive path $p$ with the fewest corners among
   those positive paths having a corner  containing  $e$.
  \end{defn}

\begin{example}
  In Figure \ref{PPforFl}, $\pi''$ is the special path    associated to the   edge $e'$, while $\pi'$ is not.
\end{example}

\begin{remark}    The special paths defined above correspond to the   partitions   of the zero rectangle or \textit{extremal}  rectangles  in \cite{MaRi}, i.e., rectangles whose length equal to $m$ or width equal to $n-m$. 
 \begin{figure}[h]
  \caption{Special paths and their corresponding   partitions for $Gr(4,7)$}\label{PTforFl}
  \bigskip
  \begin{align*}
    &\hspace{0.13cm} \pi_{1237} \hspace{1.15cm} \pi_{1267}\hspace{1.0cm} \pi_{1567}\hspace{1.01cm} \pi_{4567}\hspace{1.0cm} \pi_{3456} \hspace{1cm} \pi_{2345}\hspace{1cm} \pi_{1234}
\end{align*}

      \includegraphics[scale=0.6]{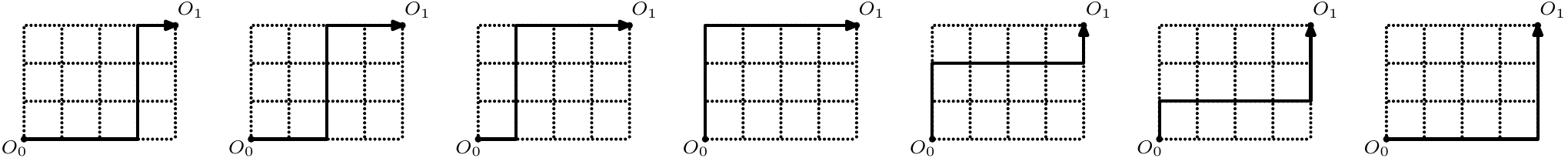}
  \begin{align*}
    &\hspace{0.33cm} \scriptstyle(3,0,0,0) \hspace{0.9cm} (3,3,0,0)\hspace{0.9cm} (3,3,3,0)\hspace{1.01cm} (3,3,3,3)\hspace{0.7cm} (2,2,2,2) \hspace{0.7cm} (1,1,1,1)\hspace{0.7cm} (0,0,0,0)
\end{align*}

          \includegraphics[scale=1]{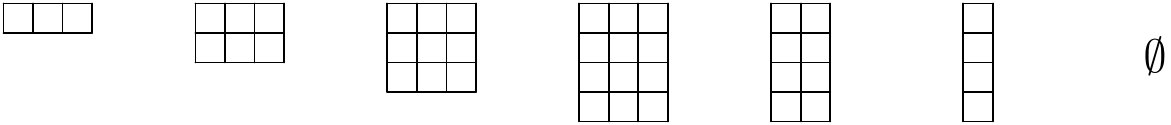}
 \end{figure}

\end{remark}

We observe that there are exactly $\sum_{i=1}^k(n_{i+1}-n_{i-1})$ edges on the roof of $\Lambda$. Furthermore
for each $1\leq i\leq k$, there are precisely $(n_{i+1}-n_{i-1})$ special paths towards $O_i$. It follows immediately that
\begin{prop-defn}\label{propAntiCano} An anti-canonical divisor of $X$ 
is given by
\begin{equation}
   -K_X:=\sum_{\pi_I\scriptsize\mbox{ is a special path}} D_{p_I}
\end{equation}
 Moreover,  $-K_X$  degenerates to the anti-canonical divisor $-K_{X_0}$ of the Gelfand-Cetlin toric variety $X_{\Delta}$ along the flat subfamily
  $\{\prod_{\pi_I\scriptsize\mbox{ is a special path}} p_I =0\}\bigcap \mathcal{X}$.
\end{prop-defn}
\noindent Here  $-K_{X_0}$ denotes the canonical representative of the anti-canonical divisor class of the toric variety  $X_{\Delta}$, given by the sum
 $\sum_{e \scriptsize \mbox{ is an effective edge of } \Lambda}  X_{F_e}$ of toric divisors.
The second statement is a direct consequence of Theorem \ref{thmToricDeg} together with the definition of  $-K_X$.

\begin{remark}
 The above $-K_X$  coincides with the anti-canonical divisor $-K'$ given by the sum of projected Richardson hypersurfaces by Knutson, Lam and Speyer \cite{KLS} in the cases of complex flag varieties and complex Grassmannians, while they are different in general. Indeed, the number of irreducible divisors in $-K_X$ is $n+n_k-n_1=\sum_{i=1}^{k}(n_{i+1}-n_{i-1})$, while that for $-K'$ is equal to $n+k-1$.
 \end{remark}

\section*{Acknowledgements}

The authors would like to thank    Cheol-Hyun Cho,    Leonardo C. Mihalcea, Ezra Miller, Ziv Ran, Vijay Ravikumar and Chi Zhang  for useful discussions and helpful comments. The authors also thank the anonymous referees for their very helpful comments.  D. Hwang  was supported by the Samsung Science and Technology
Foundation under Project SSTF-BA1602-03.   J. Lee was supported by the National Research Foundation of Korea (NRF)
grant funded by the Korea government (MSIT) (NRF-2019R1F1A1058962). 
 C. Li  was  supported by National Natural Science Foundation of China (Grants No.  11771455, 11822113, 11831017) and Guangdong Introducing Innovative and Enterpreneurial Teams No. 2017ZT07X355.

\end{document}